\documentclass[10pt]{amsart}

 \usepackage{amsmath,amsthm}
\usepackage{amssymb,latexsym}
\usepackage[all]{xy}

\makeatletter
\@addtoreset{equation}{section}
\makeatother

\newtheorem{theorem}{Theorem}[section]
\newtheorem{proposition}[theorem]{Proposition}
\newtheorem{lemma}[theorem]{Lemma}
\newtheorem{corollary}[theorem]{Corollary}

\theoremstyle{definition}

\newtheorem{remark}[theorem]{Remark}

\newcommand{\til}[1]{\tilde{#1}}

\newcommand{\id}{\operatorname{id}}
\newcommand{\ran}{\operatorname{ran}}
\newcommand{\spn}{\operatorname{span}}
\newcommand{\supp}{\operatorname{supp}}

\newcommand{\norm}[1]{\left\Vert#1\right\Vert}
\newcommand{\pcbnorm}[1]{\left\Vert#1\right\Vert_{p\mathrm{cb}}}
\newcommand{\con}{\!\ast\!}

\newcommand{\fA}{\mathcal{A}}
\newcommand{\fB}{\mathcal{B}}
\newcommand{\fC}{\mathcal{C}}

\newcommand{\fI}{\mathcal{I}}
\newcommand{\fJ}{\mathcal{J}}
\newcommand{\fM}{\mathcal{M}}
\newcommand{\fN}{\mathcal{N}}

\newcommand{\fV}{\mathcal{V}}
\newcommand{\fW}{\mathcal{W}}
\newcommand{\fX}{\mathcal{X}}

\newcommand{\Cee}{\mathbb{C}}
\newcommand{\Ree}{\mathbb{R}}
\newcommand{\En}{\mathbb{N}}

\newcommand{\alp}{\alpha}
\newcommand{\del}{\delta}
\newcommand{\Del}{\Delta}
\newcommand{\eps}{\varepsilon}
\newcommand{\gam}{\gamma}
\newcommand{\Gam}{\Gamma}
\newcommand{\lam}{\lambda}
\newcommand{\ome}{\omega}

\newcommand{\vphi}{\varphi}

\newcommand{\bl}{\mathrm{L}}
\newcommand{\ap}{\mathrm{A}_p}
\newcommand{\apg}{\mathrm{A}_p(G)}
\newcommand{\cvpg}{\mathrm{CV}_p(G)}
\newcommand{\cvp}{\mathrm{CV}_p}
\newcommand{\mat}{\mathrm{M}}
\newcommand{\kmat}{\mathrm{K}}
\newcommand{\nmat}{\mathrm{N}}
\newcommand{\pptens}{\hat{\otimes}^p}

\newcommand{\pmpg}{\mathrm{PM}_p(G)}
\newcommand{\pmp}{\mathrm{PM}_p}
\newcommand{\sop}{\mathrm{S}_0^p}
\newcommand{\sopg}{\mathrm{S}_0^p(G)}
\newcommand{\seg}{\mathrm{S}}

\begin{document}

\title[$p$-Feichtinger Segal algebras]
{$p$-Operator space structure on
Feichtinger--Fig\`{a}-Talamanca--Herz Segal algebras}

\author{Serap \"{O}ztop and Nico Spronk}

\begin{abstract}
We consider the minimal boundedly-translation-invariant Segal algebra $\sopg$
in the  Fig\`{a}-Talamanca--Herz algebra $\apg$ of a locally compact group $G$.
In the case that $p=2$ and $G$ is abelian this is the classical Segal algebra of Feichtinger.
Hence we call this the Feichtinger--Fig\`{a}-Talamanca--Herz Segal algebra of $G$.
Remarkably, this space is also a Segal algebra in $\bl^1(G)$ and is, in fact, the minimal
such algebra which is closed under pointwise multiplication by $\apg$.  Even for
$p=2$, this result is new for non-abelian $G$.
We place a $p$-operator space structure on $\sopg$, and demonstrate the naturality
of this by showing that it satisfies all natural functiorial properties:  projective
tensor products, restriction to subgroups and averaging over normal subgroups.
However, due to complications arising within the theory of $p$-operator spaces,
we are forced to work with weakly completely bounded maps in many of our results.
\end{abstract}

\maketitle

\footnote{{\it Date}: \today.


The first named author would like to thank the Scientific Projects Coordination Unit
of the Istanbul University, IRP 11488,  and the 
University of Waterloo for hosting her visit fromApril 2011 to June 2012.  
The second named author would
like thank NSERC Grant 312515-2010.}

%%Main Body

\section{Preliminaries}

\subsection{Motivation and Plan}
In \cite{feichtinger}, Feichtinger devised for any abelian group $G$,
a Segal algebra $\mathrm{S}_0(G)$ in $\bl^1(G)$ which is minimal amongst
those Segal algebras which admit uniformly bounded multiplication by characters.
Taking the Fourier transform, this may be realised as the minimal Segal algebra
in the Fourier algebra $\mathrm{A}(\hat{G})$ which admits uniformly bounded
translations.  Replacing $\hat{G}$ by $G$, for any locally compact group $G$,
and then $\mathrm{A}(G)$ by certain spaces of locally integrable functions $B$,
Feichtinger (\cite{feichtinger1}) discussed the class of minimal homogeneous
Banach spaces $B_{\min}$.  Amongst the allowable spaces discussed in
\cite{feichtinger1} are the Fig\`{a}-Talamanca--Herz algebras $\apg$, for $1<p<\infty$
of \cite{herz} and, in the abelian case, \cite{figatalamanca}.  In this paper
we discuss $\sopg=\apg_{\min}$, which we call the 
{\em $p$-Feichtinger--Fig\`{a}-Talamanca--Herz Segal algebra} of $G$, or
simply {\em $p$-Feichtinger algebra} for short.

For $p=2$, the theory of operator spaces may be applied to $\mathrm{S}_0^2(G)$,
as was done by the second named author in \cite{spronk}.
This is particularly useful because it gives, for two locally compact groups $G$ and $H$, 
a projective tensor product formula
\[
\mathrm{S}_0^2(G)\hat{\otimes}^2\mathrm{S}_0^2(H)\cong
\mathrm{S}_0^2(G\times H)
\]
where $\hat{\otimes}^2$ is the operator projective tensor product of Effros and Ruan
(\cite{effrosr}).  This, of course, is in line with their tensor product formula for preduals
of von Neuman algebras, and hence for Fourier algebras ({\em op.\ cit.})  
Losert (\cite{losert}) showed that, in general, the usual projective tensor product
of two Fourier algebras is not a Fourier algebra.

In the general case that $1<p<\infty$, various attempts have been made to
understand properties of $\apg$ via operator spaces.  See \cite{runde} and 
\cite{lambertnr}, for example.
Following the lead of Pisier (\cite{pisier}) and Le~Merdy (\cite{lemerdy}), Daws studied
properties of $\apg$ using $p$-operator spaces in \cite{daws}.  We summarise
many of Daws's results in Section \ref{ssec:popsp}.  Daws's work was followed
by An, Lee and Ruan (\cite{anleeruan}), where approximation properties were studied.  
For $p\not=2$ this theory
has many features which make it more difficult than classical operator space theory.
For example, there is a natural $p$-operator space dual structure, modelled on
the dual operator space structure of \cite{blecher}.  However, it is not, in general, the case
that the natual embedding into the second dual, $\kappa_\fV:\fV\to\fV^{**}$, is a complete
isometry.  See the summary in Proposition \ref{prop:kappainj}, below.  Even in cases
where $\kappa_\fV$ is a complete isometry,  it is not clear that a 
map $S$, for which $S^*$ is a complete isometry, is itself
is a complete quotient.  These facts forced Daws to express many results
of his as simple isometric results, and hence forced An, Lee and Ruan to do the same.
In Section \ref{ssec:wcbmaps}, we make a modest augmentation to this, and devise a theory of 
{\em weakly completely bounded} maps, hence of {\em weakly complete quotient} maps, 
to refine this theory.  In particular we see that Daws's tensor product formula, for amenable
$G$ and $H$,
\[
\apg\pptens\ap(H)\cong\ap(G\times H)
\]
is really a weakly completely isometric formula.

Many of the issues discussed above make certain matters of even defining the $p$-operator space
structure on $\sopg$ more daunting than in the $p=2$ case.  However, there is value in this
exercise as it has forced us to devise much more elementary --- though harder --- proofs, than
were found in \cite{spronk}.  In many ways, these results shed new light on the $p=2$
setting.  We justify this effort with our tensor product formula in Section \ref{ssec:tensorproducts}.
Moreover, we show the naturality of this $p$-operator space structure by demonstrating
a restiction theorem in Section \ref{ssec:restriction}, and an averaging theorem
in Section \ref{ssec:averaging}.  However, all these results live in the category
of $p$-operator spaces with morphisms of weakly completely bounded maps.

We also highlight a result which does not use operator spaces, and is new even for $p=2$
when $G$ is non-abelian.  $\sopg=\apg_{\min}$ is the minimal Segal algebra in $\bl^1(G)$
which admits pointwise multiplication by $\apg$.  This is Theorem \ref{theo:mininLone}.

\subsection{$p$-Operator spaces}\label{ssec:popsp}
We use the theory of $p$-operator spaces as presented by Daws \cite{daws}.
We shall also use the paper of An, Lee and 
Ruan~\cite{anleeruan}, and the thesis of Lee \cite{leeT}.
%and the earlier habilitation of Junge \cite{junge}.

Fix $1<p<\infty$ and let $p'$ be the conjugate index given by $\frac{1}{p}+\frac{1}{p'}=1$.
We let $\ell^p_n$ denote the usual $n$-dimensional $\ell^p$-space.
An {\em $p$-operator space structure}, on a complex vector space $\fV$, is
a family of norms $\{\norm{\cdot}_n\}_{n=1}^\infty$, each on the space
$\mat_n(\fV)$ of $n\times n$ matricies with entries in $\fV$, which satisfy
\begin{align*}
(D)&\quad\norm{\begin{bmatrix} v & 0 \\ 0& w\end{bmatrix}}_{n+m}
=\max\{\norm{v}_n,\norm{w}_n\} \\
(M_p)&\quad \norm{\alp v\beta}_n\leq\norm{\alp}_{\fB(\ell^p_n)}\norm{v}_n
\norm{\beta}_{\fB(\ell^p_n)}
\end{align*}
where $v\in\mat_n(\fV)$, $w\in\mat_m(\fV)$ and $\alp,\beta\in\mat_n$, the scalar
$n\times n$-matricies which we hereafter identify with $\fB(\ell^p_n)$.  
We will call $\fV$, endowed with a precscribed $p$-operator space
structure, a {\em $p$-operator space}.

A linear map between $p$-operator spaces $T:\fV\to\fW$ is called
{\em completely bounded} if the family of amplifications $T^{(n)}:\mat_n(\fV)\to\mat_n(\fW)$,
each given by $T^{(n)}[v_{ij}]=[Tv_{ij}]$,
is uniformly bounded, and let $\pcbnorm{T}=\sup_{n\in\En}\|T^{(n)}\|$.  Moreover
we say that $T$ is a {\em complete contraction}, or a {\em complete isometry}, if
each $T^{(n)}$ is a contraction, or, respectively, an isometry.  As proved in
\cite{pisier,lemerdy}, given a $p$-operator space $\fV$, there is a subspace $E$
of a quotient space of some $L^p$-space, and a complete isometry
$\pi:\fV\to\fB(E)$.  Here $\mat_n(\fB(E))\cong\fB(\ell^p_n\otimes^pE)$, isometrically,
where $\otimes^p$ signifies that the tensor product is normed by the identification 
$\ell^p_n\otimes^p E\cong\ell^p(n,E)$.
We shall say that $\fV$ {\em acts on $L^p$}, if there is a competely isometric
representation of $\fV$ into $\fB(L^p(X,\mu))$ for a measure space $(X,\mu)$.

We briefly review the significant structures of $p$-operator spaces, as identified by Daws.
If $\fV$ and $\fW$ are $p$-operator spaces, the space $\fC\fB_p(\fV,\fW)$ of completely bounded
maps between $\fV$ and $\fW$ is itself an operator space thanks to the isometric identifications
$\mat_n(\fC\fB_p(\fV,\fW))\cong\fC\fB_p(\fV,\mat_n(\fW))$.  Each bounded linear functional
$f$ in $\fV^*$ is automatically completely bounded with $\pcbnorm{f}=\norm{f}$, and hence
we have $\mat_n(\fV^*)\cong\fC\fB_p(\fV,\mat_n)$.  We record the following vital observations
\cite[Thm.\ 4.3 \& Prop.\ 4.4]{daws}.

\begin{proposition}\label{prop:kappainj}
{\bf (i)} If $S:\fV\to\fW$ is a complete contraction, then $S^*:\fW^*\to\fV^*$
is a complete contraction.

{\bf (ii)} A dual $p$-operator space acts on $L^p$.

{\bf (iii)} The cannonical injection $\kappa_\fV:\fV\to\fV^{**}$ is a complete
contraction, and is a complete isometry if and only if $\fV$ acts on $L^p$.
\end{proposition}

Given a vector space $\fV$ whose dual is a $p$-operator space, we let $\fV_D$
denote $\fV$ with the ``dual" $p$-operator space structure, 
i.e.\ that space which makes $\kappa_\fV:\fV\to\fV^{**}$ a complete isometry.

\begin{corollary}\label{cor:kappainj1}
Let $\fV$ and $\fW$ be $p$-operator spaces such that $\fW$ acts on $L^p$.
Then $\fC\fB_p(\fV,\fW)=\fC\fB_p(\fV_D,\fW)$ completely isometrically.
\end{corollary}

\begin{proof}
We have that 
\begin{equation}\label{eq:Sfactorise}
S=\hat{\kappa}_\fW\circ S^{**}\circ\kappa_\fV
\end{equation} 
where the left inverse
$\hat{\kappa}_\fW:\kappa_\fW(\fW)\to\fW$ is a complete isometry, by
vitrue of (iii) in the proposition above.  In other words $S$ factors through $\fV_D=\kappa_\fV(\fV)$.
It follows that $S:\fV\to\fW$ is a complete contraction exactly when
$S:\fV_D\to\fW$ is a complete contraction.  Hence
$\fC\fB_p(\fV,\fW)=\fC\fB_p(\fV_D,\fW)$ isometrically.  Replacing $\fW$ with
$\mat_n(\fW)$, for each $n$, demonstrates that this is a completely isometric
identification.
\end{proof}

The quotient structure is of particular interest to us:  
if $\fW$ is a closed subspace of $\fV$
then we identify, isometrically $\mat_n(\fV/\fW)\cong\mat_n(\fV)/\mat_n(\fW)$.  
A linear map $Q:\fV\to\fW$ is a {\em complete quotient} map if the induced
map $\til{Q}:\fV/\ker Q\to\fW$ is a complete isometry.

For convenience, we let $\fV$ and $\fW$ be complete.
Thanks to Daws \cite{daws}, we have a $p$-operator projective tensor product $\hat{\otimes}^p$.  It
obeys the usual functorial properties:  commutativity: the flip map $\Sigma:\fV\hat{\otimes}^p\fW
\to\fW\hat{\otimes}^p\fV$ is a complete isometry; duality: $(\fV\hat{\otimes}^p\fW)^*\cong
\fC\fB_p(\fV,\fW^*)$, completely isometrically; and projectivity:  if $\fV_1\subset\fV$ and
$\fW_1\subset\fW$ are closed subspaces, then $(\fV/\fV_1)\hat{\otimes}^p(\fW/\fW_1)$
is a complete quotient of $\fV\hat{\otimes}^p\fW$.  We will have occasion to consider
the non-completed dense subspace $\fV\otimes_{\wedge p}\fW$, which is
the algebraic tensor product of $\fV$ with $\fW$, with the inherited $p$-operator space structure.

Given a measure space $(X,\mu)$ we let 
\[
\nmat^p(\mu)=\fN(\bl^p(\mu))\cong\bl^{p'} (\mu)\otimes^\gam\bl^p(\mu)
\]
denote the space of nuclear operators on $\bl^p(\mu)$.  
Here, $\otimes^\gam$ denotes the projective tensor product of Banach spaces.
We note that $\nmat^p(\mu)^*\cong\fB(\bl^p(\mu))$, from which
$\nmat^p(\mu)^*$ is assigned the dual operator space structure.  We record the
following, whose proof is similar to aspects of \cite[Prop.\ 5.2]{daws} and will be omitted.

\begin{proposition}\label{prop:nucsubinj}
If $Y$ is a non-$\mu$-null subset of $X$, the $\nmat^p(\mu|_Y)$ is
a completely contractively complemented subspace of $\nmat^p(\mu)$.
Hence for any operator space $\fV$, $\nmat^p(\mu|_Y)\hat{\otimes}^p\fV$
identifies completely isometrically as a subspace of $\nmat^p(\mu)\hat{\otimes}^p\fV$.
\end{proposition}

For a $p$-operator space $\fV$, structures related to infinite matricies, $\mat_\infty(\fV)$,
and infinite matrices approximable by finite submatricies, $\kmat_\infty(\fV)$,
were worked out in \cite{leeT}, with details similar to \cite[\S 10.1]{effrosrB}.
For $S$ in $\fC\fB_p(\fV,\fW)$ we define the amplification $S^{(\infty)}:
\mat_\infty(\fV)\to\mat_\infty(\fW)$ in the obvious manner.  We observe that
$S$ is completely contractive (respectively, completely isometric) if and only of $S^{(\infty)}$ 
is contractive (respectively, isometric); and
$S$ is a complete quotient map if and only if $S^{(\infty)}|_{\kmat_\infty(\fV)}$
is a quotient map.  We will call $S$ a {\em complete surjection} when
$S^{(\infty)}|_{\kmat_\infty(\fV)}$ is a surjection.  An application of the open mapping
theorem shows that this is equivalent to having that the operators $S^{(n)}$ are uniformly
bounded below.

\subsection{Weakly completely bounded maps}\label{ssec:wcbmaps}
Various constructions that we require will not obviously respect completely
bounded maps.  However, they may be formulated with the help of a
formally more general concept.  A linear map between $p$-operator
spaces $S:\fV\to\fW$ will be called {\em weakly completely bounded}
provided that its adjoint $S^*:\fW^*\to\fV^*$ is completely bounded.
We have an obvious similar definition of a {\em weakly completely contractive}
map.   Thanks to Proposition \ref{prop:kappainj},
any complete contraction is a weakly complete contraction, and the
converse holds when $\fV$ and $\fW$ both act on $L^p$ 
(i.e.\ $\kappa_\fW\circ S=S^{**}\circ\kappa_\fV$).
We will say $S$ is a {\em weakly complete quotient} map if $S^*$ is a complete
isometry.  Thus a {\em weakly complete isometry} is an injective weakly complete
quotient map.  It is shown in \cite[Lem.\ 4.6]{daws} that a complete quotient map
is a weakly complete quotient map.  Due to the absence of a Wittstock
extension theorm --- i.e.\ we do not know if $\fB(\ell^p_n)$ is injective in the category
of $p$-operator spaces --- we do not know if a weakly complete quotient map is a 
complete quotient map, even when $\fV$ acts on $L^p$.

We say that a dual $p$-operator space $\fV^*$ {\em acts weak* on $L^p$}, if there
is a weak*-continuous complete isometry $\fV^*\hookrightarrow \fB(\bl^p(\mu))$
for some $\bl^p(\mu)$.  With this terminology, the following proposition
uses exactly the proof of \cite[Prop.\ 5.5]{daws}.

\begin{proposition}\label{prop:dualspace}
Let $\fV$ be a $p$-operator space for which $\fV^*$ acts weak* on $L^p$.
Then $\id:\fV\to\fV_D$ is a weakly complete isometry.
\end{proposition}

%The following technique is used in \cite{daws} to show that $\apg=\apg_D$ is
%a completely contractive Banach algebra. 
%t is valuable enough for us to highlight it.

%\begin{proposition}\label{prop:weakcctocc}
%Suppose $\fV=\fV_D$ and $\fW=\fW_D$ weakly completely isometrically, and
%$S:\fV\to\fW$ is a weakly completely contractive map.  Then $S:\fV_D\to\fW_D$
%\end{proposition}

%\begin{proof}
%This is fundamentally the same as the proof
%of Corollary \ref{cor:kappainj1}:  i.e.\ $S={\kappa_\fW|_{\kappa(\fW)}}^{-1}\circ S^{**}\kappa_\fV$.
%\end{proof}

Our analysis of weakly completely bounded maps will be facilitated by
some dual matrix constructions.  Let $\nmat_n^p=\fN(\ell^p_n)\cong
\ell^{p'}_n\otimes^\gam\ell^p_n$.  We let $\nmat_n(\fV)$ denote the space
of $n\times n$ matrices with entries in $\fV$, normed by the obvious identification
with $\fV\hat{\otimes}^p\nmat_n^p$.  If $T:\fV\to\fW$ is linear, we let
$\nmat_n(T):\nmat_n(\fV)\to\nmat_n(\fW)$ denote its amplification which is
identified with $T\otimes\id_{\nmat_n^p}$.  The spaces $\nmat^p_\infty$ and
$\nmat_\infty(\fV)$ are defined analagously, 
and so too is the map $\nmat_\infty(T)$, for completely bounded $T$.

\begin{lemma}\label{lem:nmatmaps}
Let $S:\fV\to\fW$ be a linear map between operator spaces.
Then the following are equivalent:

{\bf (i)} $S$ is weakly completely bounded (respectively, a weakly complete quotient map);

{\bf (ii)} there is $C>0$ such that
for each $n$ in $\En$, $\norm{\nmat_n(S)}\leq C$  %\pcbnorm{S^*}$
(respectively, $\nmat_n(S)$ is a quotient map);

{\bf (iii)} $\nmat_\infty(S)$ is defined and  bounded %$\norm{\nmat_n(S)}\leq \pcbnorm{S^*}$
(respectively, $\nmat_\infty(S)$ is a quotient map).

\noindent Moreover, the smallest value for $C$ in (ii), above, is $\pcbnorm{S^*}$.
\end{lemma}

\begin{proof} 
We have for, each $n$,  the dual space
\begin{equation}\label{eq:matndual}
\nmat_n(\fV)^*\cong(\fV\hat{\otimes}^p\nmat_n^p)^*\cong\fC\fB_p(\fV,\fB(\ell^p_n))\cong\mat_n(\fV^*)
\end{equation}
with respect to which we have identifications $\nmat_n(S)^*=S^{*(n)}$.
This gives us the immediate equivalence of (i) and (ii), as well as the minimal value of $C$ in (ii).

Proposition \ref{prop:nucsubinj} shows that each $\nmat_n(\fV)$ may be realised
isometrically as the upper left corner of $\nmat_{\infty}(\fV)$.  Let
$\nmat_{\mathrm{fin}}(\fV)=\bigcup_{n=1}^\infty\nmat_n(\fV)$, which is a dense subspace
of $\nmat_\infty(\fV)$.  Condition (ii)
gives that $\nmat_\infty(S)|_{\nmat_{\mathrm{fin}}(\fV)}$ is bounded by  $C$ %$\pcbnorm{S^*}$
(respectively, is a quotient map),
hence $\nmat_\infty(S)$ is defined and is bounded (respectively, a quotient map), i.e.\
(ii) implies (iii).  That (iii) implies (ii) is obvious.  
\end{proof}

We will say that $S:\fV\to\fW$ is a {\em weakly complete isomorphism} if $S$ is bijective and
both $S^*$ and $(S^{-1})^*$ are completely bounded.
We will further say that $S$ is a {\em weakly complete surjection}
if the induced map $\til{S}:\fV/\ker S\to\fW$ is a weakly complete isomorphism.
The following uses essentially the same proof as \cite[Cor.\ 1.2]{spronk}.  To conduct
that proof in this context, we merely
need to observe that (\ref{eq:matndual}) holds when $n=\infty$, and appeal to the infinite matrix
structures described at the end of the previous section.

\begin{corollary}\label{cor:wcsurj}
{\bf (i)}  $S$ is a weakly complete isomorphism if and only if $\nmat_\infty(S)$ is
an isomorphism.

{\bf (ii)} $S$ is a weakly complete surjection if and only if $\nmat_\infty(S)$ is surjective.
\end{corollary}

Weakly complete quotient maps play a very satisfying role with the $p$-operator
project tensor product.

\begin{proposition}\label{prop:wcqtp}
Suppose $S:\fW\to\fX$ is a weakly complete quotient map of $p$-operator spaces.
Then for any $p$-operator space $\fV$ the map $\id\otimes S$ extends
to a weakly complete quotient map from $\fV\hat{\otimes}^p\fW$ onto
$\fV\hat{\otimes}^p\fX$, which we again denote $\id\otimes S$.
If $S$ is a weakly complete isometry, then so too is $\id\otimes S$.
\end{proposition}

\begin{proof}
Under the usual dual identification, the map from $\fC\fB_p(\fV,\fX^*)$ to
$\fC\fB_p(\fV,\fW^*)$ given by $T\mapsto S^*\circ T$ is the adjoint of
$\id\otimes S:\fV\otimes_{\wedge p}\fW\to\fV\otimes_{\wedge p}\fX$.  
By assumption, $S^*$ is a complete isometry, hence so is
$T\mapsto S^*\circ T$.  It follows that $\id\otimes S$ extends to a weak
complete quotient map. If $S$ is injective, then each $\nmat_n(\id\otimes S)$ is
an isometry on $\nmat_n(\fV\otimes_{\wedge p}\fW)$, and hence extends
to an isometry on the completion.
\end{proof}

Given a family of $p$-operator spaces $\{\fV_i\}_{i\in I}$ we put a $p$-operator space
structure on the product by the identifications
$\mat_n\left(\ell^\infty\text{-}\bigoplus_{i\in I}\fV_i\right)\cong\ell^\infty\text{-}
\mat_n(\fV_i)$.  It is readily verified that $(D)$ and $(M_p)$ are satisfied.

The direct sum structure seems more subtle.  We use an approach suggested in \cite[\S 2.6]{pisierB}.
We consider for $[v_{kl}]=([v_{i,kl}])_{i\in I}$ in $\mat_n(\ell^1\text{-}\bigoplus_{i\in I}\fV_i)$
the norm 
\[
\norm{[v_{ij}]}_n=\sup\left\{\norm{\sum_{i\in I}[S_iv_{i,kl}]}_{\mat_n(\fW)}
:\begin{matrix} S_i\in\fC\fB_p(\fV_i,\fW),\,\pcbnorm{S_i}\leq 1\text{ for }i\in I \\ 
\text{ where }\fW\text{ is a }p\text{-operator space}
\end{matrix}\right\}.
\]
We observe that $\norm{[v_{ij}]}_n\leq\sum_{k,l=1}^n\sum_{i\in I}\norm{v_{i,kl}}$
and hence is finite.  Since both $(D)$ and $(M_p)$ hold in $\fW$, 
we see that this family of norms is a $p$-operator space structure
on $\ell^1\text{-}\bigoplus_{i\in I}\fV_i$.
Moreover, it is trivial to see that this space satisfies the categorical
properties of a direct sum, i.e.\ for complete contractions $S_i:\fV_1\to\fW$,
$(v_i)_{i\in I}\mapsto\sum_{i\in I}S_iv_i$ is a complete contraction.
In particular, we obtain an isometic identification
\begin{equation}\label{eq:directsum}
\fC\fB_p\left(\ell^1\text{-}\bigoplus_{i\in I}\fV_i,\fW\right)\cong
\ell^\infty\text{-}\bigoplus_{i\in I}\fC\fB_p(\fV_i,\fW).
\end{equation}
By taking $n\times n$ matricies of both sides for each $n$, we see that this is a completely isometric
identification. In particular, 
$\left(\ell^1\text{-}\bigoplus_{i\in I}\fV_i\right)^*\cong \ell^\infty\text{-}\bigoplus_{i\in I}\fV_i^*$, 
completely isometrically.  (We are indebted to M.~Daws for pointing us to this approach.)

An alternative approach is to embed
$\ell^1\text{-}\bigoplus_{i\in I}\fV_i\hookrightarrow\left(\ell^\infty\text{-}\bigoplus_{i\in I}\fV_i^*\right)^*$,
and then assign the dual operator space structure.   
If each $\fV_i$ acts on $L^p$, then each imbedding 
$\fV_i\hookrightarrow\left(\ell^1\text{-}\bigoplus_{i\in I}\fV_i\right)_D$
is a complete isometry, by virtue of Proposition \ref{prop:kappainj} (iii).  Indeed, it is clear that
each $(\fV_i)_D$ is, in turn, a complete quotient of $\left(\ell^1\text{-}\bigoplus_{i\in I}\fV_i\right)_D$.
Moreover, under the latter assumptions, $\left(\ell^1\text{-}\bigoplus_{i\in I}\fV_i\right)_D$
satisfies the categorical properties of a direct sum, for the morphisms of complete contractions,
and hence gives the operator space structure on $\ell^1\text{-}\bigoplus_{i\in I}\fV_i$ indicated above.

%We observe for a general family $p$-operator spaces 
%that each embedding $\fV_i\to\left(\ell^1\text{-}\bigoplus_{i\in I}\fV_i\right)_D$ thus factors
%through $(\fV_i)_D$ for each $i$.  From Proposition \ref{prop:kappainj} (iii) we have 
%that each $\id:\fV_i\to(\fV_i)_D$ is always a complete contraction, though it is not known generally
%to be a weakly complete isometry.  Hence, 
%it is not even clear that $\left(\ell^1\text{-}\bigoplus_{i\in I}\fV_i\right)_D$ satisfies
%the categorical properties of a direct sum in this context, even if we restrict ourselves to weakly
%completely bounded maps as morphisms.

We finally observe that $\nmat_n\left(\ell^1\text{-}\bigoplus_{i\in I}\fV_i\right)\cong
\ell^1\text{-}\bigoplus_{i\in I}\nmat_n(\fV_i)$ weakly completely isometrically.
Indeed, the dual spaces are completley isometric by virtue of 
(\ref{eq:matndual}), which gives us (\ref{eq:directsum}) with $\fW=\fB(\ell^p_n)$.

%However, if we wish a direct sum to work in the category with weakly complete contractions 
%as morphisms, we can simply define $\nmat_n\left(\ell^1\text{-}\bigoplus_{i\in I}\fV_i\right)$
%to be $\ell^1\text{-}\bigoplus_{i\in I}\nmat_n(\fV_i)$.  In effect, we have placed a ``weak $p$-operator  
%space structure" on the direct sum.  
%We do not know how to devise a $p$-operator
%space structure on $\ell^1\text{-}\bigoplus_{i\in I}\fV_i$ by which
%$\nmat_n\left(\ell^1\text{-}\bigoplus_{i\in I}\fV_i\right)\cong\nmat_n^p\hat{\otimes}^p
%\left(\ell^1\text{-}\bigoplus_{i\in I}\fV_i\right)$, isometrically for each $n$.

\subsection{Completely bounded modules and
$p$-operator Segal algebras}\label{ssec:segalalg}
Let $\fA$ be a Banach algebra which is also a $p$-operator space, and
$\fV$ be a left $\fA$-module which is also a $p$-operator space.
We say that $\fV$ is a {\em (weakly) completely bounded $\fA$-module}
if the module multiplication map $m_\fV:\fA\otimes_{\wedge p}\fV\to\fV$
is a (weakly) completely bounded map, hence extends to a (weakly) completely bounded map
$m_\fV:\fA\hat{\otimes}^p\fV\to\fV$.  In particular, for each $n$, $\nmat_n(m_\fV)
:\nmat_n(\fA\hat{\otimes}^p\fV)\to\nmat_n(\fV)$ is bounded.  Since
$\hat{\otimes}^p$ is a cross-norm, and $\nmat_{mn}(\fA\hat{\otimes}^p\fV)\cong
\nmat_m(\fA)\hat{\otimes}^p\nmat_n(\fV)$ isometrically, we 
use Lemma \ref{lem:nmatmaps} to see that for $[a_{ij}]$ in $\fA$ and
$[v_{i'j'}]$ in $\nmat_n(\fV)$ that
\begin{align}\label{eq:nmatmult}
\norm{[a_{ij}v_{i'j'}]}_{\nmat_{mn}(\fV)}
&=\norm{\nmat_{mn}(m_\fV)[a_{ij}\otimes v_{i'j'}]}_{\nmat_{mn}(\fV)} \notag \\
&\leq\pcbnorm{m_\fV^*}\norm{[a_{ij}\otimes v_{i'j'}]}_{\nmat_{mn}(\fA\hat{\otimes}^p\fV)} \\
&\leq \pcbnorm{m_\fV^*}\norm{[a_{ij}]}_{\nmat_m(\fA)}\norm{[v_{i'j'}]}_{\nmat_n(\fV)}. \notag
\end{align}
Of course, the analagous inequality characterising a completely bounded $\fA$-module is well known:
\begin{equation}\label{eq:matmult}
\norm{[a_{ij}v_{i'j'}]}_{\mat_{nm}(\fV)}\leq\pcbnorm{m_\fV}\norm{[a_{ij}]}_{\mat_n(\fA)}
\norm{v_{i'j'}]}_{\mat_m(\fV)}.
\end{equation}
We say $\fV$ is a {\em (weakly) completely contractive $\fA$-module} provided
$\pcbnorm{m_\fV}\leq 1$ ($\pcbnorm{m_\fV^*}\leq1$).
There is an obvious extension of this to right and bi-modules.
We say that $\fA$ is a {\em (weakly) completely contractive 
Banach algebra} if $\fA$ is a (weakly) completely contractive $\fA$-module over itself.

A first example of a completely contractive Banach algebra is $\fC\fB_p(\fV)=
\fC\fB_p(\fV,\fV)$ where $\fV$ is a fixed $p$-operator space.  Indeed
for $[T_{ij}]$ in $\fC\fB_p(\fV,\mat_n(\fV))$ and $[S_{i'j'}]$ in $\fC\fB_p(\fV,\mat_m(\fV))$
we have $[T_{ij}\circ S_{i'j'}]=[T_{ij}]^{(m)}\circ[S_{i'j'}]$ in $\fC\fB_p(\fV,\mat_{nm}(\fV))$
by which (\ref{eq:matmult}) is clearly satisfied.
Now we embed $\fB(\bl^p(\mu))$ into
$\fC\fB_p(\fB(\bl^p(\mu)))$ by left multiplication operators, realising $\fB(\bl^p(\mu))$
as a closed subalgebra.  It follows that $\fB(\bl^p(\mu))$ is a completely contractive Banach
algebra.  In particular, if $\fB$ is a weak*-closed algebra of $\fB(\bl^p(\mu))$, then its dual space, 
hence its natural predual, is a completely contractive $\fB$-module.

We observe that if $\fA$ and $\fB$ are weakly completely contractive Banach algebras, then
$\fA\hat{\otimes}^p\fB$ is also weakly completely contractive.  Indeed, we let
$\Sigma:\fA\hat{\otimes}^p\fB\to\fB\hat{\otimes}^p\fA$ denote the swap map, and then
we appeal to Proposition \ref{prop:wcqtp} to see that
\[
m_{\fA\hat{\otimes}^p\fB}=(m_\fA\otimes m_\fB)\circ (\id_\fA\otimes\Sigma\otimes\id_\fB)
\]
is weakly completely contractive.

Now suppose $\fI$ is a left ideal in a (weakly) completely contractive Banach algebra
$\fA$, equipped with an operator structure by which

$\bullet$ $\fI=\mat_1(\fI)$ is a Banach space,

$\bullet$ the identity injection $\fI\hookrightarrow\fA$ is (weakly) completely bounded
[contractive], and

$\bullet$ $\fI$ is a (weakly) completely bounded [contractive] $\fA$-module.

\noindent
Then we call $\fI$ a {\em (weakly) [contractive] $p$-operator Segal ideal} in $\fA$.  Given two
(weakly) $p$-operator Segal ideals in $\fA$, we write $\fI\leq\fJ$ (respectively,
$\fI\leq_w\fJ$) provided that $\fI\subset\fJ$ and the inclusion map
$\fI\hookrightarrow\fJ$ is (weakly) completely bounded.  Hence $\fI$ is a
{\em (weakly) $p$-operator Segal ideal} in $\fJ$.  A (weakly) [contractive]
$p$-operator Segal ideal $\fI=\seg\fA$ is called a {\em (weakly) [contractive] $p$-operator Segal 
algebra} in $\fA$ provided that 

$\bullet$ $\seg\fA$ is dense in $\fA$.

Suppose $\fI$ and $\fJ$ are two (weakly) $p$-operator Segal ideals of a completely 
contractive Banach algebra $\fA$, such that $\fI\cap\fJ\not=\{0\}$.  We assign
a $p$-operator space structure on $\fI\cap\fJ$ via the diagonal embedding
into the direct product space, i.e.\
$u\mapsto(u,u):\fI\cap\fJ\hookrightarrow\fI\oplus_{\ell_\infty}\fJ$.  It is straightforward
to check that $\fI\cap\fJ$ is a (weakly) Segal ideal, in this case.

\subsection{Some applications of weakly completely bounded maps}
In  \cite[Prop.\ 5.3]{daws}, the isometric identification
\begin{equation}\label{eq:nucoptp}
\nmat^p(\mu)\hat{\otimes}^p\nmat^p(\nu)\cong\nmat^p(\mu\times\nu)
\end{equation}
is given.  
This identification is one of the key points of \cite{daws}.  However, it is unkown to the authors
if it is a completely isometric identification; see further discussion in \cite[p.\ 938]{anleeruan}.

\begin{proposition}\label{prop:nucoptp}
The identification (\ref{eq:nucoptp}) is a weakly complete isometry.
\end{proposition}

\begin{proof}
Let us first assume that $\bl^p(\nu)=\ell^p_n=\bl^p(\gam_n)$,
where $\gam_n$ is the $n$-point counting measure.  Then
\begin{align*}
(\nmat^p(\mu)\hat{\otimes}^p\nmat_n^p)^*&\cong\fC\fB_p(\nmat^p(\mu),\fB(\ell^p_n))\cong
\mat_n(\nmat^p(\mu)^*) \\
&\cong\mat_n(\fB(\bl^p(\mu)))\cong\fB(\bl^p(\mu)\otimes^p\ell^p_n)\cong\fB(\bl^p(\mu\times\gam_n)).
\end{align*}
Hence the isomorphism $\nmat^p(\mu)\hat{\otimes}^p\nmat_n^p\cong\nmat^p(\mu\times\gam_n)$
of (\ref{eq:nucoptp}) is a weakly complete isometry.  By
Proposition \ref{prop:wcqtp}, the calculation above, and 
the isometric identification (\ref{eq:nucoptp}), we establish isometric
identifications
\begin{align*}
\nmat^p_n\hat{\otimes}^p\nmat^p(\mu)\hat{\otimes}^p\nmat^p(\nu)
&\cong \nmat^p(\mu\times\gam_n)\hat{\otimes}^p\nmat^p(\nu) \\
&\cong\nmat^p(\mu\times\gam_n\times\nu)\cong\nmat^p(\mu\times\nu\times \gam_n)
\cong \nmat^p_n\hat{\otimes}^p\nmat^p(\mu\times\nu).
\end{align*}
In other words, $\nmat_n(\nmat^p(\mu)\hat{\otimes}^p\nmat^p(\nu))
\cong \nmat_n(\nmat^p(\mu\times\nu))$ isometrically for each $n$.
Hence by Lemma \ref{lem:nmatmaps}, as it applies to isometries,
the identification (\ref{eq:nucoptp}) is one of a weakly complete isometry.
\end{proof}

We mildly extend some notation of \cite{daws}.  If $\fV\subset\fB(\bl^p(\mu))$ and
$\fW\subset\fB(\bl^p(\nu))$ are weak*-closed subspaces, then
$\fV\bar{\otimes}\fW$ is the weak*-closure of $\fV\otimes\fW$ in
$\fB(\bl^p(\mu)\otimes^p\bl^p(\nu))\cong\fB(\bl^p(\mu\times\nu))$.
Of course, the weak*-topology on $\fB(\bl^p(\mu))$ is given by the
identification $\fB(\bl^p(\mu))\cong\nmat^p(\mu)^*$.

\begin{proposition}\label{prop:bwtensb}
We have that $\fB(\bl^p(\mu))\bar{\otimes}\fB(\bl^p(\nu))=\fB(\bl^p(\mu\times\nu))$.
Moreover, if $\fV_0\subset\fB(\bl^p(\mu))$ and $\fW_0\subset\fB(\bl^p(\nu))$ are
subspaces with respective weak*-closures $\fV$ and $\fW$, then
$\fV_0\otimes\fW_0$ is weak*-dense in $\fV\bar{\otimes}\fW$.
\end{proposition}

\begin{proof}
We let $\Pi_\mu$ denote the set of all finite collections $\pi=\{F_1,\dots,F_{|\pi|}\}$ of pairwise
disjoint $\mu$-measurable sets such that $0<\mu(F_j)<\infty$ for each $j$.
 It is straightforward to verify that each operator
$e_\pi$ on $\bl^p(\mu)$ given by
\[
e_\pi\eta=\sum_{j=1}^{|\pi|}\int_{F_j}\eta\,d\mu\cdot\frac{1}{\mu(F_j)}1_{F_j}
\]
is a contractive projection onto $\bl^p(\mu|\pi)=\spn\{\frac{1}{\mu(F_j)}1_{F_j}\}_{j=1}^{|\pi|}$, 
a space which is isometrically isomorphic to $\ell^p_{|\pi|}$.  
Moreover, $e_\pi\fB(\bl^p(\mu))e_\pi=\fB(\bl^p(\mu|\pi))$.
We write $\pi\leq\pi'$ in $\Pi_\mu$, if each set in $\pi$ is the union of sets in $\pi'$;
making $\Pi_\mu$ into a directed set.  Then, $\lim_\pi e_\pi=I$ in the strong operator topology.
Similarly we define $\Pi_\nu$ and the associated projections on $\bl^p(\nu)$.

Thus if $T\in\fB(\bl^p(\mu\times\nu))$, for the product directed set we have
$\lim_{(\pi,\pi')}(e_{\pi}\otimes e_{\pi'})T(e_{\pi}\otimes e_{\pi'})=T$ in the weak operator
topology, and, since the net is bounded, in the weak*-topology as well.  However,
for each $(\pi,\pi')$ we have
\begin{align*}
(e_{\pi}\otimes e_{\pi'})T(e_{\pi}\otimes e_{\pi'})\in&\fB(\bl^p(\mu|\pi)\otimes^p\bl^p(\nu|\pi')) \\
&=\fB(\bl^p(\mu|\pi))\otimes\fB(\bl^p(\nu|\pi'))\subset\fB(\bl^p(\mu))\otimes\fB(\bl^p(\nu)).
\end{align*}
Hence we see that $T\in \fB(\bl^p(\mu))\bar{\otimes}\fB(\bl^p(\nu))$.

We turn now to the subspaces $\fV_0$ and $\fW_0$.  We observe that
if $T\in\fB(\bl^p(\nu))$ and $\ome\in\nmat^p(\mu)\pptens\nmat^p(\nu)$ with
norm limit $\ome=\lim_{n\to\infty}\ome_n$, where each $\ome_n\in\nmat^p(\mu)\otimes\nmat^p(\nu)$,
then $S\mapsto \langle S\otimes T,\ome\rangle=\lim_{n\to\infty}\langle S\otimes T,\ome_n\rangle$
is a norm limit of functionals associated to $\nmat^p(\mu)$, and thus itself such a functional
--- in particular this functional is weak*-continuous.  (See \cite[Lem.\ 6.1]{daws}.)  Thus
$\fV\otimes\fW_0\subset \overline{\fV_0\otimes\fW_0}^{w*}$.   Likewise
$\fV\otimes\fW\subset \overline{\fV\otimes\fW_0}^{w*}$.  Hence
$\fV\otimes\fW\subset\overline{\fV_0\otimes\fW_0}^{w*}$, and thus
$\fV\bar{\otimes}\fW\subset\overline{\fV_0\otimes\fW_0}^{w*}$.  
The converse inclusion is obvious.
\end{proof}

We note that for the predual $\fV_*=\nmat^p(\mu)/\fV_\perp$, with $\fW_*$ defined similarly,
both with either quotient or dual $p$-operator space structures, we have by \cite[Thm.\ 6.3]{daws}
that $(\fV_*\hat{\otimes}^p\fW_*)^*=\fV\bar{\otimes}_F\fW$.  Here $\fV\bar{\otimes}_F\fW$ is
a certain ``Fubini" tensor product, and contains $\fV\bar{\otimes}\fW$.

\begin{corollary}\label{cor:bwtensb1}
If $\fV\subset\fB(\bl^p(\mu))$, $\fW\subset\fB(\bl^p(\nu))$ and $\fX\subset\fB(\bl^p(\rho))$
are weak*-closed subspaces, then 
\[
(\fV\bar{\otimes}\fW)\bar{\otimes}\fX
=\fV\bar{\otimes}(\fW\bar{\otimes}\fX) 
\]
in $\fB(\bl^p(\mu)\otimes^p\bl^p(\nu)\otimes^p\bl^p(\rho))$.
\end{corollary}

\begin{proof}
We merely consider $\fV\otimes\fW\otimes\fX$ as a weak*-dense subset of 
either of the tensor closures in question.
\end{proof}

Let $G$ be a locally compact group.  As in the introduction, we let $\apg$ denote the 
Fig\`{a}-Talamanca--Herz algebra.  It is well-known to have as dual space the {\em 
$p$-pseudo-measures}
\[
\pmp=\overline{\spn\lam^p_G(G)}^{w*}
\]
where $\lam^p_G:G\to\fB(\bl^p(G))$ is the left regular representation.  We remark
that $\pmpg$ is contained in the {\em $p$-convolvers}
\[
\cvp(G)=\{T\in\fB(\bl^p(G)):T\rho^p_G(s)=\rho^p_G(s)T\text{ for }s\text{ in }G\}
\]
where $\rho^p_G:G\to\fB(\bl^p(G))$ is the right regular representation.

We recall that $\apg$ is a quotient
of $\nmat^p(G)\cong\bl^{p'}(G)\otimes^\gam\bl^p(G)$, via $P_G$, where
\[
P_G\xi\otimes\eta=\langle \xi,\lam^p_G(\cdot)\eta\rangle=\xi\ast \check{\eta}.
\]
We write $\apg_Q$ when we consider the associated quotient structure $\ran P_G$.
Thanks to \cite[Lem.\ 4.6]{daws}, the adjoint
$P_G^*:(\apg_Q)^*\cong\pmpg\to\fB(\bl^p(G))$, which is simply the injection map, is a 
complete isometry.  Hence $\pmpg$ admits completely isometric $p$-operator space
structures as $(\apg_Q)^*$, as it does as a subspace of $\fB(\bl^p(G))$.
%We write $\apg_D$ when we consider $\apg$ with the operator space structure
%inherited from $\apg\widetilde{\subset}\pmpg^*$.  
Thanks to Proposition \ref{prop:kappainj},
$\id:\apg_Q\to\apg_D$ is a complete contraction, in general; and thanks to
Proposition \ref{prop:dualspace} it is a weakly complete isometry.  
It is known to be a complete
isometry only when $G$ is amenable (\cite[Thm.\ 7.1]{daws}).  

The following is a mild augmentation of aspects of \cite[Theo.\ 7.3]{daws}.
We maintain the convention of \cite{daws} of letting $\apg=\apg_D$.

\begin{proposition}\label{prop:tensorprod}
Let $G$ and $H$ be locally compact groups.
The map $u\otimes v\mapsto u\times v:\apg\hat{\otimes}^p\ap(H)
\to \ap(G\times H)$ is a weakly complete quotient map.  It is injective
exactly when $\apg\hat{\otimes}^p\ap(H)$ is semisimple.
\end{proposition}

\begin{proof}
We note that Proposition \ref{prop:wcqtp} allows us to use the quotient operator
space structure, instead of the dual one.  Consider the diagram of maps
\[\xymatrix{
\nmat^p(G)\hat{\otimes}^p\nmat^p(H) \ar[r]^{I}\ar[d]_{P_G\otimes P_H} &
\nmat^p(G\times H) \ar[d]^{P_{G\times H}} \\
\apg_Q\hat{\otimes}^p\ap(H)_Q \ar[r]_{J}
& \ap(G\times H)_Q
}\]
where $I$ is the map from (\ref{eq:nucoptp}) and $J(u\otimes v)=u\times v$. 
%Following elementary tensors, $\xi'\otimes\xi\otimes\eta'\otimes\eta$ from
%$\bl^{p'}(G)\otimes\bl^p(G)\otimes\bl^{p'}(H)\otimes\bl^p(H)$, we see that this diagram
%commutes.  
It is easily checked, using elementary tensors $\sum_{i=1}^\infty(\xi_i\otimes\eta_i)\otimes
\sum_{j=1}^\infty(\xi_j'\otimes\eta_j')$ in $\nmat^p(G)\otimes \nmat^p(H)$, that
\[
P_{G\times H}\circ I=J\circ(P_G\otimes P_H)=J\circ(P_G\otimes \id_{\nmat^p(H)})\circ
(\id_{\nmat^p(G)}\otimes P_H)
\]
and hence
\[
\ker P_G\otimes\nmat^p(H),\;\nmat^p(G)\otimes\ker P_H\subset \ker P_{G\times H}\circ I.
\]
Thus by \cite[Prop.\ 4.10]{daws}, $\ker P_{G\times H}\circ I\supset\ker P_G\otimes P_H$, 
the diagram above commutes.  Moreover, $I$ is a weakly complete isometry, while
$P_G\otimes P_H$ and $P_{G\times H}$ are weakly complete quotient maps,
so $J$ is necessarily a weakly complete quotient map.

We observe that $\hat{\otimes}^p$ is a cross norm on $p$-operator spaces
which is easily checked to dominated the the injective norm, i.e.\ for contractive
functionals $f$ in $\fV^*$ and $g$ in $\fW^*$, $f\otimes g$ in $(\fV\pptens\fW)^*$
is contractive. Hence by \cite[Thm.\ 2]{tomiyama}, the spectrum of $\apg\hat{\otimes}^p\ap(H)$
is $G\times H$, and $J$ is the Gelfand map.  Thus $J$ is injective if and only if
$\apg\hat{\otimes}^p\ap(H)$ is semisimple.
\end{proof}

In the proof of \cite[Thm\ 7.3]{daws}, it is shown that $\pmpg\bar{\otimes}_F\pmp(H)
\subset\cvp(G\times H)$.  Thus when $\pmp(G\times H)=\cvp(G\times H)$,
the map $J$, above, is injective.  This happens when $G$ and $H$ are amenable.
In is widely suspected that $\pmp(G)=\cvp(G)$
for any $G$, but no proof nor counterexample is yet known.  The equality is shown to hold
for certain groups related to $\mathrm{SL}_2(\Ree)$ in \cite{cowling}. It is suggested
in \cite{cowling1} that when $G$ is weakly amenable, or even when $G$ posesses
the approximation property, then $\pmp(G)=\cvp(G)$.  %However, no proof is apparent.

%In the case that $G$ is discrete, see \cite[Prop.\ 6.2 \& Prop.\ 6.3]{anleeruan}.

%It seems possible this can be proved when $G$ and $H$ are semisimple Lie
%groups; see \cite{cowling}.   

We observe the following partial improvement on \cite[Theo.\ 7.3]{daws}.

\begin{corollary}
If $G$ is discrete and has the approximation property, then
\linebreak $\apg\hat{\otimes}^p\ap(H)\cong\ap(G\times H)$ weakly completely isometrically.
\end{corollary}

\begin{proof}
The results \cite[Thm.\ 5.2 \& Prop.\ 6.2]{anleeruan} tell us that $\apg$ has the $p$-operator
approximation property, in this case.  Since both of $\apg$ and $\ap(H)$ are regular and 
Tauberian, the assumptions of \cite[Theo.\ 5]{tomiyama} are satisfied, and its proof can be
modified accordingly to show that $\apg\hat{\otimes}^p\ap(H)$ is semisimple.
\end{proof}

\section{Construction of the $p$-Feichtinger Segal algebra}

\subsection{A special class of ideals}\label{ssec:specialideals}
A special class of ideals, which was defined for $p=2$ in \cite[\S 3.3]{spronk},
plays an even more important role for $\apg$ when $p\not=2$.  Thus we will discuss this
class before defining the $p$-Feichtinger Segal algebra $\sopg$.

Fix a non-null closed subset $K$ of $G$.  We let $1_K$ in $\bl^\infty(G)$
denote the indicator function of $K$, and $\bl^p(K)=1_K\bl^p(G)$.  Similarly we denote
$\bl^{p'}(K)$ and hence $\nmat^p(K)$.  Using two iterations of Proposition \ref{prop:nucsubinj},
we may identify $\nmat^p(K)$ completely isometrically as a subspace of $\nmat^p(G)$.
Let $P_K=P_G|_{\nmat^p(K)}$ and 
\[
\fM_p(K)=\ran P_K
\]
which is obviously a suspace of $\apg=\ran P_G$, though not necessarily closed.
We let $\fM_p(K)_Q$ denote this space with the quotient $p$-operator space structure.

We let $\fV_p(K)=(\ker P_K)^\perp\subset\fB(\bl^p(K))$, which is the dual space of $\fM_p(K)$.
By Proposition \ref{prop:dualspace}, $\fM_p(K)_Q=\fM_p(K)_D$ weakly completely isometrically.
We will generally take the dual structure as the default $p$-operator space structure.
We let $\lam_p^K=1_K\lam_p^G(\cdot)1_K:G\to\fB(\bl^p(K))$.
We observe that if $K$ is an open subgroup, then $\fV_p(K)\subset\pmpg$ and is, in fact
isomorphic to $\pmp(K)$.  However, for a general subset $K$, there is no reason to expect 
that $\fV_p(K)$ is in $\pmpg$, nor is an algebra.

\begin{lemma}\label{lem:vpspaces}
{\bf (i)} $\fV_p(K)$ is the weak*-closure of $\spn\lam_p^K(G)$.  

{\bf (ii)}  With dual structures, the inclusion $\fM_p(K)\hookrightarrow\apg$ is a complete
contraction.

{\bf (iii)} If $L$ is a closed subset of positive Haar measure in a locally compact group $H$,
then $\fV_p(K)\bar{\otimes}\fV_p(L)=\fV_p(K\times L)$ in $\fB(\bl^p(K\times L))$.
\end{lemma}

\begin{proof}
{\bf (i)}  Since $P_K=P_G|_{\nmat^p(K)}$ and $\apg$ is semisimple, we see that 
$\lam_p^K(G)_\perp=\ker P_K$, hence $(\spn\lam_p^K(G))_\perp=\ker P_K$.
Thus, by the bipolar theorem, $\fV_p(K)=\overline{\spn\lam_p^K(G)}^{w*}$.

{\bf (ii)} Since $\pmpg=\overline{\spn\lam_p^G(G)}^{w*}$, it then follows from (i) that 
$\fV_p(K)=\overline{1_K\pmpg 1_K}^{w*}$.  The map $T\mapsto 1_KT1_K:\pmpg\to
\fV_p(K)$ is a complete contraction, and is the adjoint of the inclusion 
$\fM_p(K)\hookrightarrow\apg$.  Thus by Proposition \ref{prop:kappainj} (iii), the inclusion
is a complete contraction.

{\bf (iii)} In $\fB(\bl^p(K\times L))=\fB(\bl^p(K)\otimes^p\bl^p(L))$ we have 
a natural identification $\spn\lam_p^{K\times L}(G)=\spn\lam_p^K(G)\otimes
\spn\lam_p^L(H)$.  The left hand side has weak*-closure $\fV_p(K\times L)$,
by (i), above.  Meanwhile, it follows (i) and Proposition \ref{prop:bwtensb}, that
the right hand side has weak*-closure $\fV_p(K)\bar{\otimes}\fV_p(L)$.
\end{proof}

We observe that (ii), above, is a generalisation  of
\cite[Prop.\ 7.2]{daws}, and, in fact, gives a simplified proof.

\begin{theorem}\label{theo:mpkideal}
For any closed non-null subset $K$ of $G$, $\fM_p(K)$ is a contractive
$p$-operator Segal ideal in $\apg$.  Moreover, $\supp\fM_p(K)\subset \overline{K^{-1}K}$,
so $\fM_p(K)$ is compactly supported if $K$ is compact.
\end{theorem}

\begin{proof}
Let $W_K$ on $\bl^p(K\times G)$ be given by
\[
W_K\eta(s,t)=\eta(s,st).
\]
Then $W_K$ is an invertible isometry with inverse $W_K^{-1}\eta(s,t)=\eta(s,s^{-1}t)$.
Thus we define a weak*-continuous complete isometry 
$\Gam_K:\fB(\bl^p(K))\to\fB(\bl^p(K\times G))$ by
\[
\Gam_K(T)=W_K(T\otimes I)W_K^{-1}.
\]
We compute, exactly as \cite[p.\ 70]{daws}, that for $t$ in $G$ we have
\begin{equation}\label{eq:gamk}
\Gam_K(\lam_p^K(t))=\lam_p^K(t)\otimes\lam_p^G(t)=\lam_p^{K\times G}(t,t).
\end{equation}
In particular, $\Gam_K(\fV_p(K))\subset\fV_p(K\times G)=\fV_p(K)\bar{\otimes}\pmpg$.

Let $\del:\fV_p(K)\bar{\otimes}\pmpg\to (\fM_p(K)\hat{\otimes}^p\apg)^*$ 
denote the inclusion map, which is weak*-continuous.  
Then we see for $u$ in $\apg$ and $v$ in $\fM_p(K)$ that
\[
\langle \del\circ\Gam_K(\lam_p^K(t)),v\otimes u\rangle
=\langle \lam_p^K(t)\otimes\lam_p^G(t),v\otimes u\rangle
=v(t)u(t).
\]
Hence the map $m:\fM_p(K)\hat{\otimes}^p\apg\to\fM_p(K)$ 
whose adjoint is $m^*=\del\circ\Gam_K$, is the multiplication map.
We note that in the notation of (\ref{eq:Sfactorise})
$m=\hat{\kappa}_{\fM_p(K)}\circ m^{**}\circ
\kappa_{\fM_p(K)\hat{\otimes}^p\apg}$ so $m$ is completely contractive.
It is shown above that the inclusion $\fM_p(K)\hookrightarrow\apg$
is completely contractive.

It is a simple computation that $\supp\xi\ast\check{\eta}\subset\overline{K^{-1}K}$
for $\eta$ in $\bl^p(K)$ and $\xi$ in $\bl^{p'}(K)$.  Hence $\supp P_K\ome
\subset\overline{K^{-1}K}$ for $\ome$ in $\nmat^p(K)$; thus the same holds for 
$v$ in $\fM_p(K)=\ran P_K$.
\end{proof}

\subsection{Construction of the $p$-Feichtinger Segal algebra}\label{ssec:construction}
We construct the minimal translation invariant Segal algebra $\sopg$ in $\apg$.
Our construction is implicit in the general setting of \cite{feichtinger1}.  However,
we strive to find a the most natural $p$-operator space structure for $\sopg$.

Let us fix, for the moment, a compactly supported $p$-operator Segal ideal $\fI$ in $\apg$.
A typical example of such an ideal is
$\mathrm{A}_p^K(G)=\{u\in\apg:\supp u\subset K\}$, for a fixed compact set $K$ with non-empty
interior, which admits the subspace $p$-operator space structure.  For any
compact $K$ of positive Haar measure, the ideal $\fM_p(K)$ from the section
above will furnish a critical example.  We define
\[
Q_\fI:\ell^1(G)\hat{\otimes}^p\fI\to\apg,\quad
Q_\fI(\del_s\otimes u)=s\ast u.
\]
We let $(\ran Q_\fI)_Q$ denote this space with the operator space structure by which $Q_\fI$
is a complete quotient map, and $(\ran Q_\fI)_D$ this space with the dual structure by which
$\kappa:(\ran Q_\fI)_D\to{(\ran Q_\fI)_Q}^{**}$ is completely isometric.  We let
\[
\sopg=(\ran Q_\fI)_D
\]
and call this the {\em $p$-Feichtinger--Fig\`{a}-Talamanca--Herz Segal algebra} of $G$,
or {\em $p$-Feichtinger Segal algebra}, for short.

\begin{theorem}\label{theo:main}
{\bf (i)} Fix a non-null compact subset $K$ in $G$ and let $\fI=\fM_p(K)$,
so $\sopg=(\ran Q_K)_D$, where $Q_K=Q_{\fM_p(K)}$.  Then, in this capacity, 
$\sopg$ is a contractive $p$-operator Segal algebra in $\apg$.

{\bf (ii)} For any two non-zero compactly supported $p$-operator 
Segal ideals $\fI$ and $\fJ$ of $\apg$,
$(\ran Q_\fI)_D=(\ran Q_\fJ)_D$ completely isomorphically.  Hence $\sopg$
is, up to complete isomorphism, independant of the $p$-operator Segal ideal.
Moreover, it is a $p$-operator Segal algebra for any such ideal.
\end{theorem}

\begin{proof}
{\bf (i)}  Using \cite[Cor.\ 2.2]{oztops}, we have an isomtetric identification
\[
\ell^1(G)\hat{\otimes}^p\fM_p(K)=\ell^1(G)\otimes^\gam\fM_p(K)
\cong\ell^1(G,\fM_p(K)).
\]
On the other hand, we obtain a completely isometric description of the dual space 
\[
(\ell^1(G)\hat{\otimes}^p\fM_p(K))^*\cong\fC\fB_p(\ell^1(G),\fV_p(K))
=\fB(\ell^1(G),\fV_p(K))%\cong\ell^\infty(G,\fM_p(G))
\]
where the last identification is furnished by \cite[Lem.\ 2.1]{oztops}.
Now the map $T\mapsto (T\del_t)_{t\in G}$ gives the isometric identification
\[
\fB(\ell^1(G),\fV_p(K))\cong\ell^\infty(G,\fV_p(K))
\til{\subset}\fB(\ell^p(G,\bl^p(K))
\]
where the latter inclusion is one on operator-valued multiplication operators:
\linebreak $(T_t)_{t\in G}(\eta_t)_{t\in G}=(T_t\eta_t)_{t\in G}$.  
This identification is a complete isometry since we have, for each $n$, isometries
\begin{align*}
\mat_n(\fB_p(\ell^1(G),\fV_p(K)))&\cong\fB_p(\ell^1(G),\mat_n(\fV_p(K)) \\
&\cong\ell^\infty(G,\mat_n(\fV_p(K))\cong\mat_n(\ell^\infty(G,\fV_p(K))).
\end{align*}
Hence, we obtain completely isometric dual space identification
\[
{(\ran Q_K)_Q}^*=(\ker Q_K)^\perp\subset  \ell^\infty(G,\fV_p(K))
\]
for which the inclusion map $(\ker Q_K)^\perp\hookrightarrow\ell^\infty(G,\fV_p(K))$
is the adjoint of $Q_K:\ell^1(G)\hat{\otimes}^p\fM_p(K)\to\ran Q_K$.
In particular, ${(\ran Q_K)_Q}^*=(\ker Q_K)^\perp$ acts weak* on $L^p$, and hence
$(\ran Q_K)_Q=(\ran Q_K)_D$ weakly completely isometrically,
thanks to Proposition \ref{prop:dualspace}.

We next let $\eps_K(s)=\bigl(\lam_p^K(t^{-1}s)\bigr)_{t\in G}$ in $\ell^\infty(G,\fV_p(K))$, 
for $s\in G$, and show that
\begin{equation}\label{eq:ranqkperp}
(\ker Q_K)^\perp=\overline{\spn\eps_K(G)}^{w*}.
\end{equation}
Indeed, we easily compute for each $(v_t)_{t\in G}$ in $\ell^1(G,\fM_p(K))$ that
\[
\left\langle\eps_K(s),(v_t)_{t\in G}\right\rangle=\sum_{t\in G}\left\langle \lam_p^K(t^{-1}s),
v_t\right\rangle=\sum_{t\in G}t\ast v_t(s)=Q_K(v_t)_{t\in G}(s)
\]
so each $\eps_K(s)\in(\ker Q_K)^\perp$ and represents evaluation at $s$.  
Since $\apg$ is semisimple, 
we have that $\eps_K(G)_\perp=\ker Q_K$, hence $(\spn\eps_K(G))_\perp=\ker Q_K$, 
and thus (\ref{eq:ranqkperp}) follows from the bipolar thoerem.

We next observe that 
\begin{equation}\label{eq:ellinftytens}
\ell^\infty(G,\fV_p(K))\cong\ell^\infty(G)\bar{\otimes}\fV_p(K)
\end{equation}
in $\fB(\ell^p(G,\bl^p(K)))\cong\fB(\ell^p(G)\otimes^p\bl^p(K))$.  Indeed let $\ell(G)$ be the
space of finitely supported functions on $G$, and similarly denote $\ell(G,\fV_p(K))$.
Then $\ell(G,\fV_p(K))\cong\ell(G)\otimes\fV_p(K)$ in 
$\fB(\ell^p(G,\bl^p(K)))\cong\fB(\ell^p(G)\otimes^p\bl^p(K))$.  Taking weak*-closure gives
(\ref{eq:ellinftytens}), thanks to Proposition \ref{prop:bwtensb}.
Thus, by additionally appealing to Corollary \ref{cor:bwtensb1},
we have weak*-continuous completely isometric isomorphisms
\[
\ell^\infty(G,\fV_p(K))\bar{\otimes}\pmpg\cong\ell^\infty(G)\bar{\otimes}\fV_p(K)\bar{\otimes}\pmpg
\cong\ell^\infty(G,\fV_p(K)\bar{\otimes}\pmpg).
\]

We recall that $\Gam_K$ is defined in (\ref{eq:gamk}).  We define 
\begin{gather*}
\Del_K:\ell^\infty(G,\fV_p(K))\to\ell^\infty(G,\fV_p(K)\bar{\otimes}\pmpg)
\cong\ell^\infty(G,\fV_p(K))\bar{\otimes}\pmpg \\
\text{by }
\Del_K(T_t)_{t\in G}=\bigl(I\otimes\lam_p^G(t)\;\Gam_K(T_t)\bigr)_{t\in G}
\end{gather*}
so $\Del_K$ is a weak*-continuous contraction.  We observe that
\begin{align*}
\Del_K(\eps_K(s))&=\bigl(I\otimes\lam_p^G(t)\;\Gam_K(\lam_p^K(t^{-1}s))\bigr)_{t\in G} \\
&=\bigl(\lam_p^K(t^{-1}s)\otimes\lam_p^G(s)\bigr)_{t\in G}\cong\eps_K(s)\otimes\lam_p^G(s)
\end{align*}
so it follows (\ref{eq:ranqkperp}) and Proposition \ref{prop:bwtensb}
that $\Del_K((\ker Q_K)^\perp)\subset(\ker Q_K)^\perp\bar{\otimes}\pmpg$.
Moreover, if we let $\del:(\ker Q_K)^\perp\bar{\otimes}\pmpg\to
\left((\ran Q_K)_D\hat{\otimes}^p\apg\right)^*$
denote the embedding, then we see for $v$ in $\ran Q_K$ and $u$ in $\apg$ that
\[
\langle \del\circ \Del_K(\eps_K(s)),v\otimes u\rangle
=\langle \del(\eps_K(s)\otimes\lam_p^G(s)),v\otimes u\rangle
=v(s)u(s).
\]
Hence, just as in the proof of Theorem \ref{theo:mpkideal}, we see that
$\sopg=(\ran Q_K)_D$ is an ideal in $\apg$ with completely contractive
multiplication.  We remark that the adjoint of the inclusion map
$\ran Q_K\hookrightarrow \apg$ is 
\[
T\mapsto (1_K\lam_p^G(t^{-1})T1_K)_{t\in G}:
\pmpg\to(\ker Q_K)^\perp\subset\ell^\infty(G,\fV_p(K))
\]
--- indeed observe that this occurs on the dense subspace $\spn\lam_p^G(G)$ --- and
this map is a complete contraction.  Hence the injection map
$(\ran Q_K)_D\hookrightarrow \apg$ is a complete contraction.

{\bf (ii)}  First, we may assume that $\fI\leq\fJ$.  Indeed, we can find $t$ in $G$ for which
$t\ast\fI\cap\fJ\not=\{0\}$. We first replace $\fI$ with $t\ast\fI$, where the latter admits a $p$-operator
structure by which $u\mapsto t\ast u:\fI\to t\ast\fI$ is a complete isometry.  We then replace $\fI$
by $\fI\cap\fJ$, where the latter has
the operator space structure suggested in Section~\ref{ssec:segalalg}, above.  
Thus the completely bounded
injection $\fI\hookrightarrow\fJ$ give a completely bounded injection $\iota:\ell^1(G)\pptens\fI
\hookrightarrow\ell^1(G)\pptens\fJ$.  We have that $Q_\fJ\circ\iota=Q_\fI$, and hence
it follows that there is a completely bounded map $\til{\iota}:(\ran Q_\fI)_Q\to(\ran Q_\fJ)_Q$.

We will first check that $\til{\iota}$ is weakly completely surjective.
Given an element $[v_{ij}]$ in $\nmat_\infty(\ran Q_\fJ)$
and $\eps>0$, we may find 
\[
\til{v}=\sum_{s\in G}\del_s\otimes[\til{v}_{ij,s}]\in\ell^p(G)\hat{\otimes}^p\nmat_\infty(\fJ)
\cong \nmat_\infty(\ell^p(G)\hat{\otimes}^p\fJ)
\]
such that
\[
\nmat_\infty(Q_\fJ)(\til{v})=[v_{ij}]\text{ and }
\norm{\til{v}}_{\nmat_n(\ell^p\hat{\otimes}^p\fJ)}=\sum_{s\in G}
\norm{\til{v}_{ij,s}]}_{\nmat_\infty(\fJ)}<\norm{[v_{ij}]}_{\nmat_\infty(\ran Q_\fJ)}+\eps.
\]
Thanks to \cite[Cor.\ 1.5]{spronk}, there are $t_1,\dots,t_n$ in $G$ and $u_1,\dots,u_n$
in $\fI$ for which $u=\sum_{l=1}^nt_k\con u_k$ satisfies $uv=v$ for $v$ in $\fJ$.  
We then have
\[
\nmat_\infty(Q_\fJ)(\til{v})=\nmat_\infty(Q_\fJ)\left(\sum_{s\in G}\del_s\otimes[u\til{v}_{ij,s}]\right)
=\sum_{k=1}^n\nmat_\infty(Q_\fJ)
\left(\sum_{s\in G}\del_{st_k}\otimes[u_k\,t_k^{-1}\ast\til{v}_{ij,s}]\right).
\]
For each $k$, we use the infite matrix version of (\ref{eq:nmatmult}) to see that
\begin{align*}
&\norm{\sum_{s\in G}\del_{st_k}\otimes
[u_k\,t_k^{-1}\ast\til{v}_{ij,s}]}_{\nmat_\infty(\ell^p\hat{\otimes}^p\fI)}
\leq \sum_{s\in G}\norm{[u_k\,t_k^{-1}\ast\til{v}_{ij,s}]}_{\nmat_\infty(\fI)} \\
&\qquad\leq C\norm{u_l}_\fI\sum_{s\in G}\norm{[\til{v}_{ij,s}]}_{\nmat_\infty(\fJ)}
<C\norm{u_l}_\fI(\norm{[v_{ij}]}_{\nmat_\infty(\ran Q_\fJ)}+\eps)
\end{align*}
where $C$ is the $p$-completey bounded norm of the adjoint of the multiplication map
$\fI\pptens\fJ\to\fI$.
Thus each $\sum_{s\in G}\del_{st_k}\otimes [u_k\,t_k^{-1}\ast\til{v}_{ij,s}]\in 
\nmat_\infty(\ell^p(G)\hat{\otimes}^p\fI)$.  It follows that that $[v_{ij}]\in \ran\nmat_\infty(Q_\fI)
=\nmat_\infty(\ran\til{\iota})$.  We then appeal to Corollary \ref{cor:wcsurj}.

We have established that $(\ran Q_\fI)_Q\cong(\ran Q_\fJ)_Q$ weakly completely
isomorphically.  Hence, we can replace $\fJ$ by $\fM_p(K)$, from (i) above.  
Thus ${(\ran Q_\fI)_Q}^*\cong(\ran Q_K)^*=(\ker Q_K)^\perp$ completely isomorphically,
and the same holds for the second duals.  Thus $(\ran Q_\fI)_D\cong(\ran Q_K)_D$
completely isomorphically.  It then easily follows that multiplication 
$m:(\ran Q_\fI)_D\hat{\otimes}^p\apg\to(\ran Q_\fJ)_D$
and inclusion $(\ran Q_\fI)_D\hookrightarrow\apg$ are both completely bounded.
\end{proof}

For a general $p$-operator Segal ideal $\fI$, in particular a contractive one, we have sacrificed
contractivity in passing to the $p$-operator Segal algebra $\sopg=(\ran Q_\fI)_D$.
This may not be necessary but will require another familiar sacrifice.
We have devised no method to show that $\sopg=(\ran Q_\fI)_D$ is a contractive 
$p$-operator Segal algebra in $\apg$, though we suspect it must be the case.

\begin{remark}\label{rem:wccsa}
If it is the case that $\fI$ is a compactly supported 
(weakly) contractive $p$-operator Segal ideal, then
$\sopg=(\ran Q_\fI)_Q$ is a weakly contractive $p$-operator Segal algebra.

Indeed, first define an operator on $\ell^1(G)\pptens\apg\cong\ell^1(G,\apg)$ by
\[
T(u_s)_{s\in G}= (s^{-1}\con u_s)_{s\in G}.
\]
We observe, as above,  
that $\ell^1(G,\apg)^*\cong\ell^\infty(G,\pmpg)\til{\subset}\fB(\ell^p(G,\bl^p(G))$.
Hence $\ell^1(G)\pptens\apg\cong\ell^1(G,\apg)_D$, weakly completely isometrically. 
We see that the map $T$, above, has adjoint $T^*(X_s)_{s\in G}=(X_s\lam_p(s^{-1}))_{s\in G}$.
Thus we see that $T$ is weakly completely contractive, in fact a weakly complete isometry.  

We now wish to show that the module multiplication map 
$m:\apg\otimes_{\wedge p}(\ran Q_\fI)_Q\to(\ran Q_\fI)_Q$ is completely contractive.  
For brevity we let
$\ap=\apg$, $\ell^1=\ell^1(G)$ and $\sop=(\ran Q_\fI)_Q$, below.
Consider the following diagram of contractions, which are the obvious
inclusion, ``shuffle" or identification maps, when not otherwise indicated.  
\[\xymatrix{
\nmat_n(\ap\pptens\ell^1\pptens\fI) \ar[r]\ar[d]_{ \nmat_n(\id\otimes Q_\fI) }
& \nmat_n(\ell^1(G,\ap)\pptens\fI) \ar[d]^{ \nmat_n(T\otimes\id) } \\
\nmat_n(\ap\pptens\sop) & \nmat_n(\ell^1(G,\ap)\pptens\fI) \ar[d] \\
\nmat_n(\ap\otimes_{\wedge p}\sop) \ar@{^{(}->}[u] \ar@{.>}[d]_{ \nmat_n(m) }
& \nmat_n(\ell^1\pptens \ap\pptens\fI) \ar[d]^{ \nmat_n(\id\otimes m_\fI) } \\
\nmat_n(\sop) & \nmat_n(\ell^1\pptens\fI) \ar[l]_{ \nmat(Q_\fI) } 
}\]
Note that $m_\fI:\ap\pptens\fI\to\fI$ is the multiplication map. 
Fix $v=[\sum_{k=1}^mu_{ij}^{(k)}\otimes v_{ij}^{(k)}]$ in 
$\nmat_n(\ap\otimes_{\wedge p}\sop)$.  For any $\eps>0$ there is
$\til{v}=\left[\sum_{k=1}^m\sum_{s\in G} u_{ij}^{(k)}\otimes\del_s\otimes v_{ij,s}^{(k)}\right]$ 
in $\nmat_n(\ap\pptens\ell^1\pptens\fI)$
with $\norm{\til{v}}_{\nmat_n(\ell^1\pptens\fI)}<\norm{v}_{\nmat_n(\ran Q_\fI)}+\eps$,
and each $v_{ij}^{(k)}=\sum_{s\in G} s\con  v_{ij,s}^{(k)}$.  If we follow
$\til{v}$, in $\nmat_n(\ap\pptens\ell^1\pptens\fI)$, along the right side of the diagram 
and back to $\nmat_n(\sop)$, we obtain $\left[\sum_{k=1}^mu_{ij}^{(k)}v_{ij}^{(k)}\right]
=m^{(n)}(v)$.
%.  We have
%\begin{align*}
%\left[\sum_{k=1}^m\sum_{s\in G} u_{ij}^{(k)}\otimes\del_s\otimes v_{ij,s}^{(k)}\right]
%&\mapsto \left[\sum_{k=1}^m\sum_{s\in G} \del_s\otimes u_{ij}^{(k)}\otimes v_{ij,s}^{(k)}\right] \\
%&\mapsto \left[\sum_{k=1}^m\sum_{s\in G} \del_s\otimes s^{-1}\con u_{ij}^{(k)}
%\otimes v_{ij,s}^{(k)}\right] \\
%&\mapsto \left[\sum_{k=1}^m\sum_{s\in G} \del_s\otimes s^{-1}\con u_{ij}^{(k)}\, v_{ij,s}^{(k)}\right] \\
%&\mapsto \left[\sum_{k=1}^m\sum_{s\in G} u_{ij}^{(k)}\, s\con v_{ij,s}^{(k)}\right] 
%=\left[\sum_{k=1}^mu_{ij}^{(k)}v_{ij}^{(k)}\right].
%\end{align*}
%The last term is precisely $m^{(n)}(v)$.  
Hence we see that 
$\norm{m^{(n)}(v)}_{\nmat_n((\ran Q_\fI)_Q)}
\leq\norm{v}_{\nmat_n(\ap\pptens(\ran Q_\fI)_Q)}+\eps$.  

We also wish to note that the inclusion $(\ran Q_\fI)_Q\hookrightarrow\apg$
is weakly completely contractive.  For any $n$ we have a contraction
$\nmat_n(\ell^1(G)\pptens\fI)\to\nmat_n(\ell^1(G)\pptens\apg)$
Moroever this contraction takes $\ker\nmat_n(Q_\fI)$ into $\ker\nmat_n(Q_{\ap})$
(here, $Q_{\ap}:$ \linebreak $\ell^1(G)\pptens\apg\to\apg$ is defined in the obvious way
and is easily checked to be a surjective complete quotient map), and hence induces a contraction 
$(\ran Q_\fI)_Q\to\apg$ which is the inclusion map.  \hfill $\Box$
\end{remark}

We show that $\sopg$ is, in essence, the minimal Segal algebra in $\apg$
closed under translations.  This requires no operator space properties.

\begin{theorem}\label{theo:mininapg}
Let $\seg\apg$ be a Segal algebra in $\apg$ which is 

$\bullet$ closed under left translations: $t\ast u\in\seg\apg$ for $t$ in $G$ and $u$ in $\seg\apg$;

$\bullet$ translations are continuous on $G$:  $t\mapsto t\ast u:G\to\seg\apg$ is continuous
for each $u$ in $\seg\apg$; and

$\bullet$ translations are bounded on $G$:  $\sup_{t\in G}\norm{t\ast u}_{\seg\ap}<\infty$
for $u$ in $\seg\apg$.

\noindent Then $\seg\apg\supset\sopg$.
\end{theorem}

\begin{proof}
By the uniform boundedness principle, the boundedness of translations
means that $\sup_{t\in G}\norm{t\ast u}_{\seg\ap}\leq C\norm{u}_{\seg\ap}$
for some constant $C$.
The assumption that $\seg\apg$ is a dense ideal in $\apg$ implies that
$\seg\apg$ contains all compactly supported elements in $\apg$; see \cite[Cor.\ 1.4]{spronk}.
Hence any compactly supported closed ideal $\fI$ of $\apg$ is contained in $\seg\apg$.
Consider $\sopg=\ran Q_\fI$.    Then for $u=\sum_{t\in G} t\ast v_t$ in $\sopg$, where
$\sum_{t\in G}\norm{v_t}_{\ap}<\infty$, we have
\[
\sum_{t\in G} \norm{t\ast v_t}_{\seg\ap}\leq C\sum_{t\in G} \norm{v_t}_{\seg\ap}
\leq CC'\sum_{t\in G} \norm{v_t}_{\ap}<\infty
\]
where $C'$ is the norm of the inclusion $\seg\apg\hookrightarrow\apg$.
Hence $\sopg\subset\seg\apg$.
\end{proof}

\subsection{The $p$-Feichtinger algebra as a Segal algebra in $\bl^1(G)$}
As before, we shall always regard $\bl^1(G)$ as a $p$-operator space by
assigning it the ``maximal operator space structure on $L^p$", as in \cite{oztops}.

A $p$-operator space $\fV$ acting on $L^p$ is a {\em completely contractive $G$-module} if there is
a unital left action of $G$ on $\fV$, $(s,v)\mapsto s\ast v$, which is continuous
on $G$ for each fixed $v$ in $\fV$ and completely contractive (hence completely
isometric) for each fixed $s$ in $G$.  Thus, for each $n$, the integrated action
$f\mapsto [f\ast v_{ij}]$ is contractive, for each $[v_{ij}]$ in $\mat_n(\fV)$,
thus completely contractive.  But then by (\ref{eq:matmult}), $\fV$ is
a completely contractive $\bl^1(G)$-module.

A Segal algebra $\seg^1(G)$ in $\bl^1(G)$ is called {\em pseudo-symmetric}
if it is closed under the group action of right translations --- $t\cdot f(s)=f(st)$
for a suitable function $f$, $t$ in $G$ and a.e.\ $s$ in $G$ --- 
and we have $t\mapsto t\cdot f:G\to\seg^1(G)$ is continuous for
$f$ in $\seg^1(G)$.   We remark, in passing, that $\seg^1(G)$ is {\em symmetric} if, moreover,
the anti-action of convolution from the right --- $f\ast t=\Del(t)t^{-1}\cdot f$ --- is an isometric
on $\seg^1(G)$.  

Consider the space $\seg^1\apg=\bl^1(G)\cap\apg$.  We assign it an operator
space structure by the diagonal embedding in the direct sum
$u\mapsto (u,u):\seg^1\apg\to (\bl^1(G)\oplus_{\ell^1}\apg)_D$.  Since each of these spaces
has the dual operator space structure, $\seg^1\apg$ injects completely contractively
into either of $\bl^1(G)$ or $\apg$.  
%The  injection
%$f\mapsto \lam_p^G(f^\star):\bl^1(G)\to\pmpg$ is contractive, hence
%completely contractive by \cite[Thm.\ 2.5]{oztops}, 
%where $f^\star=\frac{1}{\Del_G}\check{f}$.  It is simple to see that
%$\lam_p^G(f)^*=\lam_{p'}(f^\star)$.  We observe that
%\[
%f\ast(\xi\ast\check{\eta})=\langle \lam_{p'}^G(f)\xi,\lam_p^G(\cdot)\eta\rangle
%=\langle \xi,\lam_p^G(f^\star)\lam_p^G(\cdot)\eta\rangle=(\xi\ast\check{\eta})\cdot\lam_p^G(f^\star)
%\]
%where $u\cdot T$ denotes the right predual action of $T$ in $\pmpg$ on $u$ in $\apg$.  Thus,
%based on the discussion is Section \ref{ssec:segalalg}, we see that
%$\apg$ is a completely contractive $\bl^1(G)$-module.  
%Hence $(\bl^1(G)\oplus_{\ell^1}\apg)_D$
%is a completely contractive module, and $\seg^1\apg$ is a closed submodule.
This space is obviously a completely contractive $G$-module, with left translation
action $(s,u)\mapsto s\ast u$.  Hence the discussion above provides that it is
a completley contractive $\bl^1(G)$-module.
It follows that $\seg^1\apg$ is a contractive $p$-operator Segal algebra in $\bl^1(G)$.
Moreover, the injection $\apg\hookrightarrow\bl^\infty(G)$ is completely contractive,
since $\bl^\infty(G)$ has the minimal $p$-operator space structure,
which allows $\bl^1(G)$, via the predual action of multiplication by $\bl^\infty(G)$,
to be viewed as a completely contractive $\apg$-module.  Hence, 
$(\bl^1(G)\oplus_{\ell^1}\apg)_D$ is a completely contractive $\apg$-module
with $\seg^1\apg$ a closed submodule.  Thus $\seg^1\apg$ is a contractive
$p$-operator Segal algebra in $\apg$.  We call $\seg^1\apg$ the
{\em $p$-Lebesgue--Fig\`{a}-Talamanca--Herz algebra} on $G$.
The case $p=2$ is studied intensely in \cite{forrestsw}.

\begin{theorem}\label{theo:seginLone}
{\bf (i)} $\sopg$ is a pseudo-symmetric $p$-operator Segal algebra in $\bl^1(G)$.

{\bf (ii)} Given any compactly supported (weakly) $p$-operator Segal ideal $\fI$ in $\apg$,
the map $Q_\fI':\bl^1(G)\hat{\otimes}^p\fI\to\sopg$, $Q_\fI'(f\otimes u)=f\ast u$,
is a weakly complete surjection.
\end{theorem}

\begin{proof}
{\bf (i)}  It is standard that the actions of left and right translation are continuous isometries
on $\apg$, hence the actions $s\ast(\del_t\otimes u)=\del_t\otimes s\ast u$ and
$s\cdot(\del_t\otimes u)=\del_t\otimes s\cdot u$  on 
$\ell^1(G)\pptens\apg=\ell^1(G)\otimes^\gam\apg\cong\ell^1(G,\apg)$ are easily
seen to be continuous and completely isometric.  Thus if we choose $\fI=\mathrm{A}_p^K(G)$, for
a compact set $K$ with non-empty interior, we see that $\sopg=(\ran Q_\fI)_D$ has continuous
isometric translations by $G$.  From the discussion above, we see that $\sopg$
is a completely bounded $\bl^1(G)$-module

If we let $\fI$ be any compactly supported closed ideal in the pointwise algebra $\seg^1\apg$,
then $\fI$ is a contractive $p$-operator Segal ideal in $\apg$.  Then, just as 
the last part of Remark \ref{rem:wccsa}, we prove that the inclusion
$\iota:(\ran Q_\fI)_Q\hookrightarrow\bl^1(G)$ is weakly completely contractive,
and hence contractive as $\bl^1(G)$ acts on $L^p$.

{\bf (ii)}  If we follow the proof of \cite[Cor.\ 2.4 (ii)]{spronk}, we find that
$(\ran Q_\fI')_Q=(\ran Q_\fI)_Q$ weakly completely isomorphically.
\end{proof}

We do not know how to obtain a complete surjection in (ii) above.

The next result requires no operator space structure.  It is the dual result to
Theorem \ref{theo:mininapg}.  It is new, even for the case of $p=2$, for 
non-abelian $G$.

\begin{theorem}\label{theo:mininLone}
If $\seg^1(G)$ is any Segal algebra in $\bl^1(G)$ for which the pointwise
multiplication satisfies 
$\apg\cdot\seg^1(G)\subset\seg^1(G)$, then $\seg^1(G)\supset\sopg$.
\end{theorem}

\begin{proof}
Given any compact set $K\subset G$ we can arrange a compactly supported ideal
$\fI$ in $\apg$ which contains a function which is identically $1$ on $K$.
Hence we can arrange such an ideal $\fI$ for which $\fI\cdot\seg^1(G)\not=\{0\}$.
Our assumption on $\seg^1(G)$ provides that $\fI\cdot\seg^1(G)\subset\seg^1(G)$.
Since $\sopg$ is pseudo-symmetric we have that $\sopg\ast(\fI\cdot\seg^1(G))\subset\sopg$.
Indeed, if $u\in\sopg$ and $f\in \fI\cdot\seg^1(G)$, then
\[
u\con f=\int_Gu\con t\, f(t)\,dt
\]
which may be regarded as a Bochner integral in $\sopg$, as
\[
\norm{u\con f}_{\sop}\leq\sup_{t\in\supp\fI}\norm{u\con t}_{\sop}\int_{\supp\fI}|f(t)|\,dt
<\infty.
\]
Hence $u\con f\in\sopg$.  Also $\sopg\ast(\fI\cdot\seg^1(G))\subset
\sopg\ast\seg^1(G)\subset\seg^1(G)$.  Using the fact that $\sopg$ contains a bounded approximate
identity for $\bl^1(G)$, we see that $\{0\}\not=\sopg\ast(\fI\cdot\seg^1(G))$.  Thus we see that
$\{0\}\not=\sopg\ast(\fI\cdot\seg^1(G))\subset \sopg\cap\seg^1(G)$.  In particular
$\sopg\cap\seg^1(G)$ is a Segal algebra in $\apg$ which, being a Segal algebra in $\bl^1(G)$,
satisfies the conditions of Theorem \ref{theo:mininapg}.  Hence 
$\seg^1(G)\supset\sopg\cap\seg^1(G)\supset\sopg$.
\end{proof}

\section{Functiorial Properties}

\subsection{Tensor products}\label{ssec:tensorproducts}
The realisation of the tensor product formula is, perhaps, the most significant reason
to consider the operator space structure on $\sopg$.  In Proposition \ref{prop:tensorprod},
we improved on the formula \cite[Thm.\ 7.3]{daws} only mildly, i.e.\ we acheived
that $\apg\pptens\ap(H)\cong\ap(G\times H)$ weakly isometrically when
$G$ and $H$ are amenable, say, rather than simply isometrically.
Hence we should not expect, at the present time, to do better with $\sopg\pptens\sop(H)$.

In \cite[Thm.\ 3.1]{spronk}, the injectivity of $\mathrm{PM}_2(G)$ for almost connected $G$
was put to good use.  Lacking any such property for $p\not=2$, we are forced
to return to the special ideals of Section \ref{ssec:specialideals}.

\begin{theorem}\label{theo:tensprod}
The map $u\otimes v\mapsto u\times v:\sopg\pptens\sop(H)\to
\sop(G\times H)$ is a weakly complete surjection.  It is a bijection
whenever the extended map $\apg\pptens\ap(H)\to\ap(G\times H)$
is a bijection.
\end{theorem}

\begin{proof}
We fix non-null compact subsets $K$ in $G$ and $L$ in $H$.  Consider, first,
the following commutating diagram, where all ideals $\fM_p$ and all algebras $\ap$
have the quotient or dual operator space structure, which we know to be weakly
completely isomorphic to one another.
\[\xymatrix{
\nmat^p(K)\pptens\nmat^p(L) \ar[rr]^i \ar[dd]_{P_K\otimes P_L} \ar@{_{(}->}[dr]
& & \nmat^p(K\times L) \ar@{_{(}->}[dr] \ar'[d][dd]^{P_{K\times L}} & \\
& \nmat^p(G)\pptens\nmat^p(H) \ar[rr]^{I\phantom{mm}} \ar[dd]_<<<<<<{P_G\otimes P_H} & &
\nmat^p(G\times H) \ar[dd]^{P_{G\times H}} \\
\fM_p(K)\pptens\fM_p(L) \ar'[r][rr]_<<<<j \ar[dr]_{\iota_K\otimes\iota_L} & &
\fM_p(K\times L) \ar@{_{(}->}[dr]_{\iota_{K\times L}} & \\
& \apg\pptens\ap(H) \ar[rr]_J & & \ap(G\times H)
}\]
Here weakly complete isometries
$i$ and $I$ are from (\ref{eq:nucoptp}), $j$ and $J$ are given on their respective
domains by $u\otimes v\mapsto u\times v$, and $\iota_K$, $\iota_L$ and $\iota_{K\times L}$
are completely contractive inclusion maps.  The diagonal 
inclusion maps on the top are complete isometries
thanks to Proposition \ref{prop:nucsubinj}.

The inclusion $\ker P_G\otimes P_H\subset \ker P_{G\times H}\circ I$ is noted in the proof
of Proposition \ref{prop:tensorprod}.  Thus we have
\begin{align*}
\ker P_K\otimes P_L&=(\ker P_G\otimes P_H)\cap\left(\nmat^p(K)\pptens\nmat^p(L)\right)  \\
&\subset (\ker P_{G\times H}\circ I)\cap\left(\nmat^p(K)\pptens\nmat^p(L)\right) =\ker P_{K\times L}\circ i.
\end{align*}
Thus $j:\fM_p(K)\pptens\fM_p(L)\to \fM_p(K\times L)$ is a weakly complete quotient map.
Moreover, we see from this that $j$ is injective provided that $J$ is injective.

Now consider the diagram
\[\xymatrix{
\ell^1(G)\pptens\fM_p(K)\pptens\ell^1(H)\pptens\fM_p(L)\ar[r]^<<<<S\ar[d]^{Q_K\otimes Q_L}
&\ell^1(G\times H)\pptens\fM_p(K\times L)\ar[d]_{Q_{K\times L}} \\
\sopg\pptens\sop(H) \ar[r]^{\til{\jmath}} &\sop(G\times H)
}\]
where $S=(\id\otimes j)\circ(\id\otimes\Sigma\otimes \id)$, so $S$ is a complete quotient map,
and $\til{\jmath}(u\otimes v)=u\times v$.  If we consider $\sopg=\ran Q_K$,
then we know that $\sopg_Q=\sopg_D$ weakly completely isometrically.  
Since $Q_K\otimes Q_L$ and $Q_{K\times L}$ are complete quotient maps, 
and 
\[
\ker Q_K\otimes \ell^1(H)\otimes \fM_p(L),\;\ell^1(G)\otimes\fM_p(K)\otimes\ker Q_L
\subset \ker Q_{K\times L}\circ S
\]
it follows that $\til{\jmath}$ is a complete quotient map in this diagram.  Now if $J$
is injective, and thus so too is $j$, we note that 
\[
J\circ Q_K\otimes Q_L=
Q_{K\times L}\circ S :\ell^1(G)\otimes\fM_p(K)\otimes\ell^1(H)\otimes\fM_p(L) 
\to \ap(G\times H).
\]
Thus, by taking closures, it follows that $\ker Q_K\otimes Q_L=\ker Q_{K\times L}\circ S$.
Hence $\til{\jmath}$ is injective in this case.

If we consider $\sopg=\ran Q_\fI$ for an arbitrary compactly supported $p$-operator 
Segal ideal, then we 
obtain that $\til{\jmath}$ is a weakly complete surjection or weakly complete isomorphism, depending
on injectivity of $J$.
\end{proof}

%A straighforward modification of the proof above, with $H$ playing the role of $L$, hence
%$\ap(H)$ the role of $\fM_p(L)$, etc., yields the following result.  Recall that the definition of the
%ideal $\fM_p(K\times H)$ makes no requirment that $K\times H$ be compact.

%\begin{corollary}\label{cor:tensprod1}
%The map $u\otimes v\mapsto u\times v:\sopg\pptens\ap(H)\to
%\fM_p(G\times H)$ is a complete surjection.  It is a complete bijection
%whenever the extended map $\apg\pptens\ap(H)\to\ap(G\times H)$
%is a bijection.
%\end{corollary}

We remark that the identity operator on $\apg\otimes\ap(H)$ extends to a contraction
$\apg\otimes^\gam\ap(H)\to\apg\pptens\ap(H)$.  We cannot guarantee that this map is
injective.  In the case that $p=2$,% and neither of $G$ nor $H$ are not almost abelian,
we do not know if this map is injective, unless one of the component Fourier algebras
has the approximation property; say in the case that one of $G$ or $H$ is abelain or compact.
When $p\not=2$, then even if both $G$ and $H$ are abelain, we still do not know if this map
is injective.   In the second case, the map is unlikely to be surjective.   However, we have no proof.
If $G$ is discrete, it is trivial to verify that $\sopg=\ell^1(G)$ completely isomorphically; indeed
choose $\fI=\Cee 1_{\{e\}}$ in the construction of $\sopg$.  In this case we have
$\sopg\pptens\sop(H)=\sopg\otimes^\gam\sop(H)$.  It seems likely that this is the only 
situation in which this tensor formula holds.

\subsection{Restriction to subgroups}\label{ssec:restriction}
Let $H$ be a closed subgroup of $G$.  We briefly recall that the restriction map
$u\mapsto u|_H:\mathrm{A}_2(G)\to\mathrm{A}_2(H)$ is a complete quotient map
since its adjoint is a certain $*$-homomorphism $\mathrm{PM}_2(H)\hookrightarrow
\mathrm{PM}_2(G)$, hence a complete isometry.  When $p\not=2$, the fact that there
is a natural complete isometry $\pmpg\hookrightarrow\pmp(H)$ is not automatic, and must be
verified.  Thankfully, \cite{derighetti} provides a proof which is easily modifiable for our needs.

Fix a Bruhat function $\beta$ on $G$ (\cite[Def.\ 8.1.19]{reiters}) and let
$q(x)=\int_H\beta(xh)\frac{\Del_G(h)}{\Del_H(h)}\,dh$ for $x$ in $G$ for a fixed Haar measure
on $H$.  Then, by \cite[(8.2.2)]{reiters}, there exists a quasi-invariant integral 
$\int_{G/H}\dots dxH$, such that we have an invariant integral on $G$,
given for $f\in\fC_c(G)$ by
\begin{equation}\label{eq:haarintegral}
\int_G f(x)\, dx = \int_{G/H}\int_H\frac{f(xh)}{q(xh)}\,dh\,dxH.
\end{equation}
By dominated convergence, this formula will hold for any compactly supported integrable function $f$.
%We observe that $q$, above, can be replaced by $qt$ for any $t$ in $G$ where $qt(x)=q(tx)$,
%the achieve the same Haar integral.  Indeed, this is an immediate consequence of
%left invariance of $\int_G \dots dx$.

We will use the isometry $U=U_G:\bl^p(G)\to\bl^p(G)$, given by $Uf=\check{f}\frac{1}{\Del_G^{1/p}}$
which satisfies $U^{-1}=U$ and $U\lam_p^G(\cdot)U=\rho_p^G$, where $\rho_G$ is the right
regular representation given by $\rho_p^G(s)f(t)=f(st)\frac{1}{\Del_G(t)^{1/p}}$.
Thus if $\cvp'(G)=\{T\in\fB(\bl^p(G)):T\lam_p^G(\cdot)=\lam_p^G(\cdot)T\}$, we have
$U\cvpg U=\cvp'(G)$.  The map $T\mapsto UTU:\cvp(G)\to\cvp'(G)$ is evidently a 
weak*-continuous complete isometry for which $U\pmpg U=\overline{\spn\rho_p^G(G)}^{w*}$.

The following is a modification of \cite{derighetti} and \cite{delaported}.  See
also the treatment in \cite{derighettiB}.
We let $R_H:\apg\to\ap(H)$ denote the restriction map.  Its existence was first established
in \cite{herz2}, with a simplified proof given in \cite{delaported}.  

\begin{theorem}\label{theo:cvpinjection}
There is a complete isometry $\iota:\cvp(H)\to\cvpg$ such that $\iota|_{\pmp(H)}=R_H^*$.
In particular, $R_H:\apg\to\ap(H)$ is a weakly complete quotient map and a complete contraction.
\end{theorem}

\begin{proof}
In order to make use of (\ref{eq:haarintegral}) as stated, we replace $\cvpg$ with
$\cvp'(G)$ and $\lam_p^G$ with $\rho_p^G$.

For a function $\vphi$ on $G$ and $x$ in $G$ we let $\vphi x(s)=\vphi(xs)$.
Now for $T$ in $\cvp'(H)$ and $\vphi$ and $\psi$ in $\fC_c(G)$, we observe that
the function
\[
x\mapsto\left\langle \left.\left(\frac{\psi}{q^{1/{p'}}}\right)x\right|_H,
T\left[\left(\left.\frac{\vphi}{q^{1/p}}\right)x\right|_H\right]\right\rangle
=\int_H \frac{\psi(xh)}{q(xh)^{1/{p'}}}T\left[\left(\left.\frac{\vphi}{q^{1/p}}\right)x\right|_H\right](h)\, dh
\]
is constant on cosets $xH$, since $T\in\cvp'(H)$.  Thus we may define $\iota(T)$ by
\[
\langle\psi,\iota(T)\vphi\rangle
=\int_{G/H}\left\langle\left.\left(\frac{\psi}{q^{1/{p'}}}\right)x\right|_H,
T\left[\left(\left.\frac{\vphi}{q^{1/p}}\right)x\right|_H\right]\right\rangle \, dxH.
\]
The fact that $|\langle\psi,\iota(T)\vphi\rangle|\leq 
\norm{\psi}_{\bl^{p'}(G)}\norm{T}_{\fB(\bl^p(H))}\norm{\vphi}_{\bl^p(G)}$ will
follow from a computation below.  Hence $\iota(T)$ defines a bounded
operator on $\bl^p(G)$.  Let us consider $[T_{ij}]$ in $\mat_n(\cvp'(H))$.  We observe that
for $[\vphi_i]$ and $[\psi_j]$ in the column space $\fC_c(G)^n$ that an application of
H\"{o}lder's inequality  and the usual operator norm inequality
%fact that
%$\norm{
% [T_{ij}]\left[\left(\left.\frac{\vphi_i}{q^{1/p}}\right)x\right|_H\right]}_{\ell^{p}(n,\bl^{p}(H))}
% \leq \norm{[T_{ij}]}_{\mat_n(\fB(\bl^p(H)))}
% \norm{\left[\left(\left.\frac{\vphi_i}{q^{1/p}}\right)x\right|_H\right]}_{\ell^{p}(n,\bl^{p}(H))}$
give
\begin{align*}
&\left|\bigl\langle [\psi_j],\iota(T)^{(n)}[T_{ij}][\vphi_i]\bigr\rangle\right|
=\left|\int_{G/H}\left\langle \left[\left.\left(\frac{\psi_j}{q^{1/{p'}}}\right)x\right|_H\right],
[T_{ij}]\left[\left(\left.\frac{\vphi_i}{q^{1/p}}\right)x\right|_H\right]\right\rangle \, dxH\right| \\
&\qquad\leq \left(\int_{G/H}\norm{
\left[\left.\left(\frac{\psi_j}{q^{1/{p'}}}\right)x\right|_H\right]}_{\ell^{p'}(n,\bl^{p'}(H))}^{p'}
dxH\right)^{1/p'} \\
&\qquad\qquad\qquad\qquad \left(\int_{G/H}\norm{
 [T_{ij}]\left[\left(\left.\frac{\vphi_i}{q^{1/p}}\right)x\right|_H\right]}_{\ell^{p}(n,\bl^{p}(H))}^{p}
dxH\right)^{1/p} \\
&\qquad\leq 
\left(\sum_{j=1}^n \int_{G/H}\int_H \frac{|\psi_j(xh)|^{p'}}{q(xh)}\, dh\, dxH\right)^{1/p'} \\
&\qquad\qquad\qquad\qquad\norm{[T_{ij}]}_{\mat_n(\fB(\bl^p(H)))} 
\left(\sum_{i=1}^n \int_{G/H}\int_H \frac{|\vphi_i(xh)|^p}{q(xh)}\, dh\, dxH\right)^{1/p} \\
&\qquad =\norm{[\psi_j]}_{\ell^{p'}(n,\bl^{p'}(G))}\norm{[T_{ij}]}_{\mat_n(\fB(\bl^p(H)))}
\norm{[\vphi_i]}_{\ell^p(n,\bl^p(G))}.
\end{align*}
This shows that $\norm{\iota^{(n)}[T_{ij}]}_{\mat_n(\fB(\bl^p(G)))}\leq
\norm{[T_{ij}]}_{\mat_n(\fB(\bl^p(G)))}$, hence $\iota:\cvp'(H)\to\fB(\bl^p(G))$
is a complete contraction. 

From this point, the proof of \cite[pps.\ 73--75]{derighetti} 
or of \cite[\S7.3, Theo.\ 2]{derighettiB} can be followed nearly
verbatim, with $[\psi_j]$ and $[\vphi_i]$ in $\fC_c(G)^n$ replacing
$\psi$ and $\vphi$, and the norms from $\ell^{p'}(n,\bl^{p'}(G))$ and
$\ell^p(n,\bl^p(G))$ on these respective columns in place of usual scalar norms; and $T$ in
$\cvp'(H)$ playing the role of $\rho_p^G(\mu)$ (which is denoted $\lam_G^p(\mu)$
by that author).  Hence
we have that $\iota:\cvp'(H)\to\fB(\bl^p(G))$ is a complete isometry.

It is shown in \cite[\S 7.1 Theo.\ 13]{derighettiB} that $\iota(T)\in\cvp'(G)$.

Accepting differences between our notation and theirs, it
is shown in both \cite{delaported} and \cite[\S7.8, Theo.\ 4]{derighettiB} that
$R_H^*=\iota|_{U_H\pmp(H)U_H}(\cdot)$.  Thus it follows from Lemma 
\ref{lem:nmatmaps} that $R_H:\apg\to\ap(H)$
is a weakly complete quotient map, with either quotient or, thanks to 
Proposition \ref{prop:dualspace}, dual operator space structure.  It follows from
the factorisation $R_H
=\hat{\kappa}_{\ap(H)}\circ(\iota(U_H\cdot U_H))^{**}\circ\kappa_{\apg}$
that $R_H$ is a complete contraction  (with dual operator space structure).  

To get the ``left" version, as in our statment of theorem, we simply replace
$\iota$ by $U_G\iota(U_H\cdot U_H)U_G$.
\end{proof}

%For $\vphi$ and $\psi$ in $\fC_c(G)$, $t$ in $G$ and $T$ in $\cvp'(H)$,
%we use  the Haar integral formula (\ref{eq:haarintegral}) and the comments following it
%to compute that
%\begin{align*}
%\langle \psi,\iota(T)\lam_p^G(t)\vphi\rangle 
%&= \int_{G/H}\int_H \frac{\psi(xh)}{q(xh)^{1/{p'}}}
%T\left[\left(\left.\frac{\lam_p^G(t)\vphi}{q^{1/p}}\right)x\right|_H\right](h) \, dh\,dxH \\
%&= \int_{G/H}\int_H  \frac{\psi(xh)}{q(xh)^{1/{p'}}}
%T\left[\left(\left.\frac{\vphi}{(qt)^{1/p}}\right)t^{-1}x\right|_H\right](h) \, dh\,dxH \\
%&= \int_{G/H}\int_H \frac{\psi(txh)}{q(txh)^{1/{p'}}}
%T\left[\left(\left.\frac{\vphi}{(qt)^{1/p}}\right)x\right|_H\right](h) \, dh\,dxH \\
%&=\int_{G/H}\left\langle\left.\left(\frac{\lam_{p'}(t^{-1})\psi}{(qt)^{1/{p'}}}\right)x\right|_H,
%T\left[\left(\left.\frac{\vphi}{(qt)^{1/p}}\right)x\right|_H\right]\right\rangle \, dxH \\
%&=\langle \lam_{p'}^G(t^{-1})\psi,\iota(T)\vphi\rangle=\langle\psi,\lam_p^G(t)\iota(T)\vphi\rangle.
%\end{align*}
%Hence $\iota(T)\in\cvp'(G)$.

Unfortunately, we cannot determine if $R_H:\apg\to\ap(H)$ is a complete quotient map, even
when both groups are amenable.  

The class of ideals $\fM_p(K)$ of
Section \ref{ssec:specialideals} will play a special role in obtaining a restriction theorem
on the $p$-Feichtinger algebra.
We let $(\fM_p(K)|_H)_Q$ denote $\fM_p(K)|_H$ with the operator space making this space
a complete quotient of $\fM_p(K)$ via $R_H$.  We then place on $\fM_p(K)|_H$ the dual
operator space structure, i.e.\  $\fM_p(K)|_H=(\fM_p(K)|_H)_D$.

\begin{lemma}\label{lem:idealresttoH}
Let $K$ be a nonnull closed set in $G$.  Then $(\fM_p(K)|_H)_Q=(\fM_p(K)_H)_D$
weakly completely isometrically.  Moreover, $\fM_p(K)|_H=(\fM_p(K)|_H)_D$ is
a contractive operator Segal ideal in $\ap(H)$.
\end{lemma}

\begin{proof}
Since $\lam_p^G(h)=\iota(\lam_p^H(h))$, for $h$ in $H$, we have that 
\[
\ker (R_H|_{\fM_p(K)})^\perp=\overline{1_K\iota(\pmp(H))1_K}^{w*}
\]
a space we hereafter denote by $\fV_p^H(K)$.   Hence $(\fM_p(K)|_H)^*
\cong \fV_p^H(K)$, and it follows from Proposition \ref{prop:dualspace}
that $(\fM_p(K)_H)_Q=(\fM_p(K)_H)_D$ weakly completely isometrically.

The proof of Lemma \ref{lem:vpspaces}
shows that $\fV_p^H(K\times G)=\fV_p^H(K)\bar{\otimes}\iota(\pmp(H))$ in
$\fB(\bl^p(G\times G))=\fB(\bl^p(G)\otimes^p\bl^p(G))$.  Moreover, substituting
$h$ in $H$ for $t$ in (\ref{eq:gamk}) we see that $\Gam_K(\iota(\pmp(H)))
\subset \fV_p^H(K)\bar{\otimes}\iota(\pmp(H))$.  
A straightforward adaptation of the proof of Theorem \ref{theo:mpkideal} shows
that $\fM_p(K)|_H=(\fM_p(K)|_H)_D$ is a contractive operator Segal ideal in $\ap(H)$.
\end{proof}

Not knowing that $R_H:\apg\to\ap(H)$ is a complete surjection, we can hardly
expect to do better for $\sopg$.

\begin{theorem}\label{theo:restriction}
The restriction map $R_H:\sopg\to\sop(H)$ is a
weakly complete surjection and is completely bounded.
\end{theorem}

\begin{proof}
The first part of the proof of \cite[Thm.\ 3.3]{spronk} can be adapted directly
to show that $R_H(\sop(G))\subset\sop(H)$.  Thus it remains to show that
$R_H$ is a weakly complete surjection.  We consider the Segal ideals $\fM_p(K)$
of $\apg$ and $\fM_p(K)|_H$ of $\ap(H)$, from the lemma above.  In particular
$R_H:\fM_p(K)\to\fM_p(K)|_H$ is a weakly complete quotient map.  Thus, he following
diagram commutes, and has surjective top row and right column.
\[\xymatrix{
\nmat_\infty(\ell^1(H)\pptens\fM_p(K)) \ar[rr]^{ \nmat_\infty(\id\otimes R_H) }
\ar[d]_{ \nmat_\infty(Q_K|_{ \ell^1(H)\pptens\fM_p(K) }) } 
& & \nmat_\infty(\ell^1(H)\pptens\fM_p(K)|_H)\ar[d]^{ \nmat_\infty(Q_{\fM_p(K)|_H})} \\
\nmat_\infty(\sopg) \ar[rr]^{\nmat_\infty(R_H)} & & \nmat_\infty(\sop(H))
}\]
Hence the bottom row is surjective, and we appeal to Corollary \ref{cor:wcsurj}
to see that $R_H:\sopg\to\sop(H)$ is weakly completely surjective.

Since we assign the dual $p$-operator space structure on the range space $\sop(H)$,
the usual argument, i.e.\ modelled after the factorisation (\ref{eq:Sfactorise}), shows
that $R_H$ is completely bounded.
\end{proof}

\subsection{Averaging over subgroups}\label{ssec:averaging}
Given a closed normal subgroup $N$ of $G$, we consider the
averaging map $\tau_N:\fC_c(G)\to\fC_c(G/N)$ given by
$\tau_N f(sN)=\int_N f(sn)\,dn$.  With appropriate scaling of Haar measures, 
$\tau_N$ extends to a homomorphic quotient map from $\bl^1(G)$ to $\bl^1(G/N)$.
We wish to study the effect of $\tau_N$ on $\sopg$.  We first require the following
result which will play a role similar to that of \cite[Prop.\ 3.5]{spronk}.

\begin{lemma}\label{lem:mpkapgnmod}
Let $K$ be a non-null closed subset of $G$.  Then $\fM_p(K)$ is a completely
contractive $\ap(G/N)$-module under pointwise product, i.e.\ the product
$uv(s)=u(s)v(sN)$ for $u$ in $\fM_p(K)$ and $v$ in $\ap(G/N)$.
\end{lemma}

\begin{proof}
This is a simple modification of the proof of Theorem \ref{theo:mpkideal}.  Indeed, we replace
$W_K$ in that proof by
$W_K^{G/N}$ on $\bl^p(K\times G/N)$, given by
\[
W_K^{G/N}\eta(s,tN)=\eta(s,stN)
\]
and then replace $\Gam_K$ with $\Gam_K^{G/N}:\fB(\bl^p(K))\to\fB(\bl^p(K\times G/N))$, given by
\[
\Gam_K^{G/N}(T)=W_K^{G/N}(T\otimes I)(W_K^{G/N})^{-1}.
\]
The rest of the proof holds verbatim.
\end{proof}

We now get an analogue of \cite[Thm.\ 3.6]{spronk}. 
The neccessity to consider only weakly completely
bounded maps will arise in various aspects of the proof below.

\begin{theorem}\label{theo:averaging}
We have that $\tau_N(\sop(G))\subset\sop(H)$ and $\tau_N:\sopg\to\sop(H)$
is a weakly complete surjection, hence a completely bounded map.
\end{theorem}

\begin{proof}
The proof is essentially that of \cite[Thm.\ 3.6]{spronk}.  Unfortunately,
in order to highlight the aspects which require modification, we are forced to 
revisit nearly every aspect of that proof.  We do, however, take liberty
to omit some computational details which are simple modifictions of those in the aforementioned 
proof.

First, we fix a non-null compact set $K$ and
show that $\tau_N:\fM_p(K)\to\ap(G/N)$ is completely bounded.  

We show that $\tau_N(\fM_p(K))\subset\ap(G/N)$. To see this we 
have that
\[
\norm{\tau_N}_{\fB(\bl^{p'}(K),\bl^{p'}(G/N)}
\leq\inf\left\{\sup_{s\in G}\tau_N(|\vphi|^p)(sN)^{1/p}:\vphi\in\fC_c(G),\vphi|_K=1\right\}<\infty.
\]
Of course, this estimate makes sense with roles of $p$ and $p'$ interchanged.
Motivated by the computation of \cite[p.\ 187]{lohoue} which shows that for compactly supported 
integrable $\eta$ we have
$\tau_N(\check{\eta})=\left[\Del_{G/N}\tau_N(\check{\Del}_G\eta)\right]^\vee$, we define
$\theta_N:\bl^p(K)\to\bl^p(G/N)$ by $\theta_N(f)=\Del_{G/N}\tau_N(\check{\Del}_Gf)$.
Then indeed $\theta_N$ admits the claimed codomain with
\[
\norm{\tau_N}_{\fB(\bl^p(K),\bl^p(G/N)}
\leq\sup_{s\in K}\Del_{G/N}(sN)\sup_{t\in K}\Del_G(t^{-1})\norm{\tau_N}_{\fB(\bl^p(K),\bl^p(G/N)}
<\infty.
\]
Hence if $\xi\in\bl^{p'}(K)$ and $\eta\in\bl^p(K)$ we have that
\[
\tau_N(\xi\ast\check{\eta})=\tau_N(\xi)\ast\tau_N(\check{\eta})
=\tau_N(\xi)\ast\left[\theta_N(\eta)\right]^\vee\in\ap(G/N)
\]
and it follows that $\tau_N(\fM_p(K))\subset\ap(G/N)$.  Moreover, the support of 
$\tau_N(\fM_p(K))$ is contained in the image of $K^{-1}K$ in $G/N$ and is thus compact.

We now want to see that $\tau_N:\fM_p(K)\to\ap(G/N)$ is indeed completely bounded
We first observe that with the column space structure $\bl^p(G)_c$, and row space structure
$\bl^{p'}(G)_r$ of \cite{lemerdy}, we have a weakly completely isometric identification
\[
\nmat^p(G)=\bl^{p'}(G)_r\pptens\bl^p(G)_c
\]
thanks to \cite[Prop.\ 2.4 \& Cor.\ 2.5]{anleeruan}.  Moreover, \cite[Prop.\ 2.4]{anleeruan}
shows that $\tau_N:\bl^{p'}(K)_r\to\bl^{p'}(G/N)_r$ and
$\theta_N:\bl^p(K)_c\to\bl^p(G/N)_c$ are completely bounded, and hence
\[
\tau_N\otimes\theta_N:\bl^{p'}(K)_r\pptens\bl^p(K)_c
\to\bl^{p'}(G/N)_r\pptens\bl^p(G/N)_c
\]
is completely bounded, and thus forms a weakly completely bounded map
$\tau_N\otimes\theta_N:\nmat^p(K)\to\nmat^p(G/N)$.  Hence we consider the 
following commuting diagram.
\[\xymatrix{
\nmat^p(K)\ar[rr]^{\tau_N\otimes\theta_N}\ar[d]_{P_K} & & \nmat^p(G/N) \ar[d]^{P_{G/N}} \\
\fM_p(K) \ar[rr]_{\tau_N} &  & \ap(G/N)
}\]
Since the top arrow is a weakly complete isometry, and the down arrows are both
(weakly) complete quotient maps, we see that the bottom arrow must be a weakly
complete contraction.

We now place on $\tau_N(\fM_p(K))$ the operator space structure which makes
$\tau_N:\fM_p(K)\to\tau_N(\fM_p(K))$ a complete quotient map, hence a weakly complete
quotient map.  We wish to see that, in this capacity, $\tau_N(\fM_p(K))$ is a
weakly contractive $p$-operator Segal ideal in $\ap(G/N)$.

First, let $\pi_N:G\to G/N$ be the quotient map.  We observe that for
$u$ in $\ap(G/N)$ and $v$ in $\fM_p(K)$ we have
\[
\tau_N(v)(sN)u(sN)=\int_N v(sn)u\circ\pi_N(sn)\, dn
=\tau_N(v\,u\circ \pi_N)(sN)
\]
so $\tau_N(v)u=\tau_N(v\, u\circ \pi_N)\in\tau_N(\fM_p(K))$, as $u\circ \pi_N\; v\in\fM_p(K)$.
Now consider the following commuting diagram where $m$ is the 
completely contractive multiplication map promised by Lemma \ref{lem:mpkapgnmod}
and $\til{m}$ is the multiplication map promised above.
\[\xymatrix{
\fM_p(K)\otimes_{\wedge p}\ap(G/N) \ar[rr]^<<<<<<<<<<<<m 
\ar[d]_{\tau_N\otimes\id} & & \fM_p(K) \ar[d]^{\tau_N} \\
\tau_N(\fM_p(K))\otimes_{\wedge p}\ap(G/N) \ar[rr]_<<<<<<<<<<{\til{m}} & & \tau_N(\fM_p(K))
}\]
Since the top arrow is a complete contraction, and the down arrows are weak complete
quotient maps, the bottom arrow must be a complete contraction as well, hence
extends to $\tau_N(\fM_p(K))\pptens\ap(G/N)$.

From Theorem \ref{theo:seginLone} (ii) we have that the map
$Q'_{\fM_p(K)}:\bl^1(G)\pptens\fM_p(K)\to\sopg$ is a weakly complete
surjection.  Similarly, appealing also to the fact that $\tau_N(\fM_p(K))$
is a compactly supported weakly $p$-operator Segal ideal in $\ap(G/N)$, we have that
$Q'_{\tau_N(\fM_p(K))}:\bl^1(G/N)\pptens\tau_N(\fM_p(K))\to\sop(G/N)$
is a weakly complete surjection.  

The following diagram commutes, where the down arrows are weakly complete surjections
by virtue of Theorem \ref{theo:seginLone} (ii), and, additionally, the fact that $\tau_N(\fM_p(K))$
is a compactly supported weakly $p$-operator Segal ideal in $\ap(G/N)$.
\[\xymatrix{
\bl^1(G)\pptens\fM_p(K) \ar[rr]^{\tau_N\otimes\tau_N}\ar[d]_{Q'_{\fM_p(K)} }
& & \bl^1(G/N)\pptens\tau_N(\fM_p(K)) \ar[d]^{Q'_{\tau_N(\fM_p(K))}} \\
\sopg \ar[rr]_{\tau_N} & & \sop(G/N)
}\]
Since the top arrow is a complete quotient map, 
hence a weakly complete surjection, and the down arrows are weakly
complete surjections, the same must hold of the bottom arrow.
\end{proof}

\section{Discussion}

\subsection{On containment relations}
Let $1<q< p\leq 2$ or $2\leq p<q<\infty$.  If $G$ is amenable, then
$\apg\subset\mathrm{A}_q(G)$ contractively.  See \cite[Thm.\ C]{herz2}
and \cite[Remark, p.\ 392]{runde1}.  Thus we have $\mathrm{A}_p^K(G)\subset
\mathrm{A}_q^K(G)$ contractively (these ideals are defined in Section \ref{ssec:construction}),
and hence it follows that $\sopg\subset
\mathrm{S}_0^q(G)$, boundedly.

Since $\bl^p(G)$ is, isomorphically, a quotient of a subspace of a $L^q$-space,
$\pmp(G)$ can be endowed with a $q$-operator space structure, and thus so can
$\apg$.  Hence, {\em is the inclusion $\apg\subset\mathrm{A}_q(G)$ completely bounded?}

If the answer to the above question is true, even in the weakly complete sense, then
the adjoint gives a completely bounded map $\mathrm{PM}_q(G)\to\pmpg$.
Thus, in the notation of the proof of Theorem \ref{theo:main} we should be able to prove that
there is a completely bounded map $\ell^\infty(G,\fV_q(G))\to
\ell^\infty(G,\fV_p(G))$, which, when restricted to $(\ker Q_K)^\perp$, is the adjoint
of the inclusion $\sopg\hookrightarrow\mathrm{S}_0^q(G)$.  Hence we would see that
the latter inclusion is weakly completely bounded.

\subsection{Fourier transform}
Let $G$ be abelian with dual group $\hat{G}$.  In \cite{feichtinger} it is
shown that $\mathrm{S}_0^2(G)\cong\mathrm{S}_0^2(\hat{G})$, via the Fourier
transform $F$.  Suppose $p\not=2$.
{\em Is there a meaningful intrinsic characterisation of $F(\sopg)$ as a subspace of
$\mathrm{A}_2(\hat{G})$?}

%\vfill

%\pagebreak

Addresses:
\linebreak
 {\sc 
 Istanbul University, Faculty of Science, Department of Mathematics, 34134 Vezneciler, Istanbul, Turkey. \\
Department of Pure Mathematics, University of Waterloo,
Waterloo, ON\quad N2L 3G1, Canada.}

\medskip
Email-adresses:
\linebreak
{\tt oztops@istanbul.edu.tr}
\linebreak {\tt nspronk@uwaterloo.ca}


\begin{thebibliography}{10}

\bibitem{anleeruan}
G.~An, J.J.~Lee and Z.-J.~Ruan. 
\newblock On p-approximation properties for p-operator spaces,
\newblock J. Funct. Anal. 259 (2010) 933--974.

\bibitem{blecher}
D.P.~Blecher.
\newblock The standard dual of an operator space,
\newblock Pacific Math. J. 153 (1992) 15--30.

\bibitem{cowling}
M.~Cowling.
\newblock La synth\`{e}se des convoluteurs de $L^p$ de certains groupes pas moyennables. 
\newblock  Boll. Un. Mat. Ital. A (5) 14 (1977) 551--555. 

\bibitem{cowling1}
M.~Cowling.
\newblock The predual of the space of convolutors on a locally compact group. 
\newblock  Bull. Austral. Math. Soc. 57 (1998) 409--414.


\bibitem{daws}
M.~Daws.
\newblock $p$-operator spaces and Fig\`{a}-Talamanca--Herz algebras,
\newblock J. Operator Theory 63 (2010) 47--83.

%\bibitem{decanneireh}
%J.~de Canni\`{e}re and U.~Haagerup.
%\newblock Multipliers of the Fourier algebras of some simple Lie groups and their 
%discrete subgroups. 
%\newblock {\em Amer. J. Math.}  107:455--500, 1985.

\bibitem{delaported} 
J.~Delaporte and A.~Derighetti.
\newblock On Herz' extension theorem,
\newblock Boll. Un. Mat. Ital. A (7) 6 (1992) 245Ð247.

\bibitem{derighetti}
A.~Derighetti.
\newblock Relations entre les convoluteurs d'un groupe localement compact et ceux d'un 
sous-groupe ferm\'{e},
\newblock  Bull. Sci. Math. (2) 106 (1982) 69Ð84. 


\bibitem{derighettiB}
A.~Derighetti.
\newblock Convolution operators on groups,
\newblock Lecture Notes of the Unione Matematica Italiana, 11. Springer, Heidelberg; UMI, 
Bologna, 2011.


\bibitem{effrosr}
E.G.~Effros and Z.-J.~Ruan.
\newblock On approximation properties for operator spaces,
\newblock Internat. J. Math. 1 (1990) 163--187.


\bibitem{effrosrB}
E.G.~Effros and Z.-J.~Ruan.
\newblock Operator Spaces,
\newblock London Math. Soc. Monogr. (N.S.), vol. 23, Oxford Univ. Press, New York, 2000.

\bibitem{feichtinger}
H.-G.~Feichtinger.
\newblock On a new Segal algebra,
\newblock Montash. Math. 92 (1981) 269--289.

\bibitem{feichtinger1}
H.-G.~Feichtinger. 
\newblock A characterization of minimal homogeneous Banach spaces,
\newblock Proc. Amer. Math. Soc. 81 (1981) 55--61.

\bibitem{figatalamanca}
A.~Fig\`{a}-Talamanca.
\newblock Translation invariant operators in $L^p$,
\newblock Duke Math. J. 32 (1965) 495--501.

\bibitem{forrestsw}
B.E.~Forrest, N.~Spronk  and P.J.~Wood.
\newblock Operator Segal algebras in Fourier algebras,
\newblock Studia Math. 179 (2007) 277--295.


%\bibitem{hewittrII}
%E. Hewitt and K.~A. Ross, {\em Abstract Harmonic Analysis II},
%volume 152 of {\em Grundlehern
%  der mathemarischen Wissenschaften}.
%Springer, New York, 1970.

\bibitem{herz}
C.~Herz.
\newblock Une g\'{e}n\'{e}ralisation de la notion de transform\'{e}e de Fourier--Stieltjes,
\newblock Ann. Inst. Fourier 24 (1974) 145--157.


\bibitem{herz1}
C.~Herz.
\newblock The theory of $p$-spaces with an application to convolution operators,
\newblock Trans. Amer. Math. Soc. 154 (1971) 69--82.

\bibitem{herz2}
C.~Herz.
\newblock Harmonic synthesis for subgroups,
\newblock Ann. Inst. Fourier 23 (1973) 91--123.


\bibitem{lambertnr}
A.~Lambert, M.~Neufang and V.~Runde.
\newblock Operator space structure and amenability for Fig\`{a}-Talamanca--Herz algebras,
\newblock J. Funct. Anal. 211 (2004) 245--269.


\bibitem{leeT}
J.-J.~Lee. 
\newblock On $p$-Operator Spaces and their Applications,
\newblock Thesis (Ph.D.), University of Illinois at Champaign-Urbana, 2010


\bibitem{lemerdy}
C.~Le Merdy.
\newblock Factorization of $p$-completely bounded multilinear maps,
\newblock Pacific J. Math. 172 (1996) 187--213.


\bibitem{lohoue}
N.~Lohou\'{e}.
\newblock Remarques sur les ensembles de synth\`{e}se des alg\`{e}bras de groupe 
localement compact,
\newblock J. Funct. Anal. 13 (1973) 185--194.


\bibitem{losert}
V.~Losert.
\newblock On tensor products of the Fourier algebras,
\newblock Arch. Math. 43 (1984) 370--372.


\bibitem{oztops}
S.~\"{O}ztop and N.~Spronk.
\newblock Minimal and maximal $p$-operator space structures, 
\newblock Canadian Math. Bull. (accepted 2012), {\tt 10.4153/CMB-2012-030-7}, 10 pages.


\bibitem{pisier}
G.~Pisier.
\newblock Completely bounded maps between sets of Banach space operators,
\newblock Indiana Univ. Math. J. 39 (1990) 249--277.

\bibitem{pisierB}
G.~Pisier.
\newblock Introduction to Operator Space Theory, 
\newblock London Math. Soc. Lec. Notes Ser., v. 294, Cambridge Univ. Press, 2003.

\bibitem{reiters}
H.~Reiter and J.D.~Stegeman.
\newblock Classical Harmonic Analysis and Locally Compact Groups,
\newblock London Math. Soc. (N.S.), vol. 22, Oxford Univ. Press, New York, 2000.

\bibitem{runde}
V.~Runde.
\newblock Operator Fig\`{a}-Talamanca--Herz algebras,
\newblock Studia Math. 155 (2003) 153--170. 

\bibitem{runde1}
V.~Runde.
\newblock Representations of locally compact groups on QSL$_p$-spaces and a $p$-analog of the 
Fourier-Stieltjes algebra,
\newblock Pacific J. Math. 221 (2005) 379--397.

\bibitem{spronk}
N.~Spronk.
\newblock Operator space structure on Feichtinger's Segal algebra,
\newblock J. Funct. Anal. 248 (2007) 152--174.


\bibitem{tomiyama}
J.~Tomiyama.
\newblock Tensor products of commutative Banach algebras,
\newblock T\^{o}hoku Math. J. (2) 12 1960 147--154. 


\end{thebibliography}
\end{document}